\documentclass[11pt,a4paper,twoside]{amsart}

\usepackage{amssymb,amsfonts,amsthm}
\usepackage[leqno]{amsmath}
\usepackage[utf8]{inputenc}
\usepackage[final,allcolors=blue,colorlinks=true]{hyperref}
\usepackage[hmargin=2.5cm,vmargin=3.5cm,twoside]{geometry}

\usepackage[shortalphabetic,initials]{amsrefs}

\renewcommand{\MR}[1]{\null}

\newtheorem{theorem}{Theorem}[section]
\newtheorem{lemma}[theorem]{Lemma}
\newtheorem{proposition}[theorem]{Proposition}

\theoremstyle{definition}

\newtheorem{remark}[theorem]{Remark}
\newtheorem*{remark*}{Remark}

\numberwithin{equation}{section}

\newcommand{\abs}[1]{{\lvert#1\rvert}}
\newcommand{\norm}[1]{{\lVert#1\rVert}}

\newcommand{\N}{\mathbb{N}}
\newcommand{\R}{\mathbb{R}}
\newcommand{\Rn}{\mathbb{R}^n}
\newcommand{\Rnk}{\R^{n-k}}
\newcommand{\Ry}{\Rk_y}
\newcommand{\Rk}{\R^k}

\newcommand{\Rt}{\R_t}
\newcommand{\Ha}{\mathcal{H}}
\newcommand{\Hn}{\mathcal{H}^n}
\newcommand{\Hk}{\mathcal{H}^k}
\newcommand{\Hnk}{\mathcal{H}^{n-k}}

\newcommand{\Omegas}{\Omega^\sigma}
\newcommand{\Es}{E^\sigma}

\newcommand{\esspj}{\pi^+_{n-k}}
\newcommand{\esspjt}{\pi^+_{n-k,t}}
\newcommand{\proj}{\pi_{n-k}}
\newcommand{\projt}{\pi_{n-k,t}}

\newcommand{\loc}{{\mathrm{loc}}}
\newcommand{\Ds}{D^\mathrm{s}}
\newcommand{\Da}{D^\mathrm{a}}

\newdimen\mex
\newcommand{\nquad}{\mspace{-18.0mu}}
\newcommand{\nqquad}{\mspace{-36mu}}

\makeatletter
\def\niv{\mathrel{\hbox{\hglue -0.4\mex
  \vrule \@height 1\mex \@width .1\mex
  \vrule \@height .1\mex \@width 1\mex
  \hglue -0.2\mex}}}
\makeatother

\newcommand{\rest}{\niv}

\newcommand{\us}{u^\sigma}
\newcommand{\Su}{\mathcal{S}_u}
\newcommand{\Sus}{\mathcal{S}_{u^\sigma}}
\newcommand{\mS}{\mathcal{S}}

\newcommand{\sobo}{W^{1,1}_{0,y}}

\newcommand{\frec}{f_\infty}
\newcommand{\ftrec}{\tilde{f}_\infty}

\newcommand{\weakto}{\rightharpoonup}

\def\Xint#1{\mathchoice
   {\XXint\displaystyle\textstyle{#1}}%
   {\XXint\textstyle\scriptstyle{#1}}%
   {\XXint\scriptstyle\scriptscriptstyle{#1}}%
   {\XXint\scriptscriptstyle\scriptscriptstyle{#1}}%
   \!\int}
\def\XXint#1#2#3{{\setbox0=\hbox{$#1{#2#3}{\int}$}
     \vcenter{\hbox{$#2#3$}}\kern-.5\wd0}}

\def\dashint{\Xint-}

\newcommand{\brid}{\partial^*}
\newcommand{\bmis}{\partial^{\mathcal{M}}}
\newcommand{\measbdry}{\bmis}
\newcommand{\aev}{\text{-a.e.\ }}
\newcommand{\ft}{\tilde{f}}
\newcommand{\Ft}{\tilde{F}}

\DeclareMathOperator{\supp}{spt}
\DeclareMathOperator{\essup}{ess\,sup}
\newcommand{\step}[1]{\noindent\textbf{Step #1.}}

\let\phi\varphi
\let\epsilon\varepsilon

\newcommand{\Ath}{\tilde{A}_h}
  
\begin{document}

\title{The Steiner rearrangement in any codimension}

\author[G. M. Capriani]{Giuseppe Maria Capriani}
\address{Dipartimento di Matematica e Applicazioni \\ Università degli Studi di Napoli ``Federico
II''\\ Via Cintia -- 80126, Napoli, Italy}
\email{giuseppe.capriani@gmail.com}

\date{\today}

\begin{abstract}
We analyze the Steiner rearrangement in any codimension of Sobolev and $BV$ functions. In particular, we prove a
Pólya-Szegő inequality for a large class of convex integrals. Then, we give minimal assumptions under which
functions attaining equality are necessarily Steiner symmetric.
\end{abstract}
\maketitle

\section{Introduction}
Symmetrization techniques are a powerful tool to deal with those variational problems whose extrema are expected to
exhibit symmetry properties due either to the geometrical or to the physical nature of the problem (see, for instance,
the classical book \cite{PS} and \cite{kawohl}).

As in the the isoperimetric theorem, it is well-known that the perimeter of a set decreases under several
types
of symmetrizations such as polarization, standard Steiner symmetrization or more general Steiner
symmetrization with respect to a $n-k$ dimensional plane.

Similarly, the so-called Pólya-Szegő inequality states that a large class of Dirichlet-type
integrals depending on the gradient of a real-valued function decreases under rearrangement operations such as the
Schwarz spherical rearrangement or the Steiner rearrangement in codimension $k$, see Definition \ref{eq:steiner}.

In this framework, a natural question, which has been extensively studied in recent years, is to give a characterization
of the equality cases in the Pólya-Szegő
inequality as well as in inequalities concerning symmetrization of sets.

In the celebrated paper \cite{BZ} Brothers and Ziemer characterized the equality cases in the Pólya-Szegő inequality
for the Schwarz rearrangement of a Sobolev function under the minimal assumption that the set of critical points of the
rearranged function has zero Lebesgue measure (see also \cite{FV} for an alternative proof). The corresponding
inequality for $BV$ functions  was first proved in \cite{hilden}, while a much finer
analysis is carried out in \cite{bv-rearra}, where also the equality cases are characterized.

Concerning the standard Steiner symmetrization and its higher codimension version, the validity of the isoperimetric
inequality and of the Pólya-Szegő principle are also well-known, see for instance a proof via polarization
given in \cite{brock-s} and the references therein. On the other hand, the characterization of the equality cases seems
to be a much harder problem. The first result in this direction was proved in \cite{CCF} in connection to the 
perimeter inequality for the standard Steiner symmetrization. In analogy to what was pointed out in \cite{BZ},
also in this case it turns out that such characterization may hold only
under the assumption that the boundary of the set is almost nowhere orthogonal to the symmetrization hyperplane.
However this condition alone is not yet enough and a connectedness assumption, in a suitable measure theoretic sense,
must
be required on the set.

The equality cases in the Pólya-Szegő inequality for the standard Steiner rearrangement of Sobolev and $BV$
functions were investigated in \cite{CF}. Again, the crucial assumption was that the set where the  derivative of the
extremal function in the direction orthogonal to the hyperplane of symmetrization vanishes is negligible. As 
for sets, also some connectedness and geometrical assumptions have to be made on the domain supporting the
function.

Very recently, in \cite{barcafus} the equality cases in the perimeter inequality for the Steiner symmetrization in
codimension $k$ were characterized using a different approach from the one in \cite{CCF}, aimed to reduce the
problem to a careful study of the barycentre of the sections of the original set.

We further develop the analysis made in the above papers by considering the
Pólya-Szegő
inequality for Dirichlet-type integrals of Sobolev functions or area-like integrals of $BV$
functions. First we prove the Pólya-Szegő inequality for general convex integrands $f$ depending
on the gradient of a Sobolev function $u$. Besides convexity, we assume that $f$ is
non-negative, vanishes at $0$ and depends on the norm of the $y$-component of the gradient of $u$, $y\in\Rk$
being the direction of symmetrization.

In order to characterize the equality cases, i.e., to show that $u$ coincides with its Steiner rearrangement $\us$ up to
translations, the strict convexity of $f$ is required together with the assumption that $\nabla_y \us\neq 0$ a.e.. Note
that the result is false if one of the two previous assumptions is dropped. As in \cite{CF}, suitable assumptions on the
domain $\Omega$ of $u$ are also needed.

A similar analysis on the Pólya-Szegő inequality and on the characterization of the equality cases is also carried out
in the more general framework of functions of bounded variation. In this case, however, one has to assume that 
$f$ has linear growth at infinity and to suitably extend the integral by taking into account the singular part of the
gradient measure $Du$, see \eqref{eq:Jf}.

These results are proved via geometric measure theory arguments based on the isoperimetric
theorem,  the coarea formula and fine properties of Sobolev and $BV$ functions. In particular, to deal with the $BV$
case one has to rewrite the original functional, which in principle depends on $Du$, as a functional
defined on the graph of $u$ and depending on the generalized normal to the graph.

The latter approach could be also carried out in the Sobolev case and therefore we could have chosen to deal from the
 beginning with $BV$ functions and then to deduce the Sobolev case as a corollary. However, we have preferred to give
in the Sobolev case an independent proof that avoids the heavy machinery required in the $BV$ case.

It is also worth to mention that, though the general strategy follows the path set up in previous papers, namely
\cite{CCF} and \cite{CF}, we have to face here an extra substantial difficulty which appears only when dealing with the
Steiner rearrangement in codimension strictly larger than $1$. This difficulty appears for those functions that Almgren
and Lieb, in \cite{AL}, called \textit{coarea irregular} (see the discussion at the end of Section \ref{sec:states}).
These functions, which can even be of class $C^1$, are precisely the ones where Schwarz rearrangement in discontinuous
with respect to the $W^{1,p}$ norm.

Finally, the paper is organized as follows. In Section \ref{sec:states} we state and comment the main results and in
Section \ref{sec:background} we collect some background material on sets of finite perimeter and functions of bounded
variation. Section \ref{sec:sobolev} is devoted to Sobolev functions while Section \ref{sec:bv} deals with $BV$
functions and functionals depending on the normal.

\section{Statement of the main results}\label{sec:states}
Given two sets $E$ and $F$, we denote the \emph{symmetric difference} by $E\triangle F:=(E\cup F)\setminus(E\cap
F)$. Given two open sets $\omega\subset\Omega$ we write $\omega\Subset\Omega$ if $\omega$ is \emph{compactly contained}
in $\Omega$, i.e., if $\overline{\omega}\subset\Omega$ and $\overline{\omega}$ is compact.
Let $n\geq 2$ and $1\leq k<n$. We write a generic point $z\in\R^n$ as $z=(x,y)$, where
$x\in\Rnk$ and $y\in\Rk$. In order to clarify the different roles of the variables we will also write
$\R^n=\Rnk\times\Ry$ and $\R^{n+1}=\Rnk\times\Ry\times\Rt$.

Given a measurable set $E\subset\Rnk\times\Rk$, for $x\in\Rnk$ we define the section of $E$ at $x$ as
\begin{equation}\label{section}
E_x:=\left\{y\in\Rk:(x,y)\in E \right\}.
\end{equation}
Then we define the \emph{projection} of $E$ as
\begin{equation}\label{proj}
 \proj(E):=\left\{x\in\Rnk:(x,y)\in E \right\}
\end{equation}
and the \emph{essential projection} as
\begin{equation}\label{esspj}
\esspj(E):=\left\{x\in\Rnk:(x,y)\in E,\,L(x)>0\right\},
\end{equation}
where $L(x):=\Hk(E_x)$ and $\Hk$ is the
$k$-dimensional Hausdorff measure. We define the \emph{Steiner symmetral (in codimension $k$)} $E^{\sigma}$
of $E$ as
\begin{equation}\label{steiner_symmetral}
E^{\sigma}:=\left\{(x,y)\in\Rnk\times\Rk:x\in\esspj({E}),\,\abs{y}
^k\leq\frac {L(x)}{ \omega_k } \right\} ,
\end{equation}
where $\omega_k$ is the volume of the $k$-dimensional ball.

When $E\subset\Rnk\times\Ry\times\Rt$, its Steiner symmetral $E^\sigma$ is defined in the same way, after
replacing \eqref{section}--\eqref{steiner_symmetral} by similar definitions. In particular, we set
\begin{gather*}
 E^\sigma:=\left\{(x,y,t)\in\Rnk\times\Ry\times\Rt:(x,t)\in\esspjt({E}),\,\abs{y}
^k\leq\frac {L(x,t)}{ \omega_k } \right\}\\
\esspjt(E):=\left\{(x,t)\in\Rnk\times\Rt:(x,y,t)\in E,\,L(x,t)>0\right\}\,,
\end{gather*}
where $L(x,t):=\Ha^{k+1}(E_{x,t})$ and $E_{x,t}:=\{y\in\Rk:(x,y,t)\in E\}$. 

Given a non-negative measurable function $u$ defined on $E$ such that for $\Hnk\aev x\in\esspj(E)$
\begin{equation}\label{eq:misura_sezioni}
\Hk\left(\{y\in E_x : u(x,y)>t\}\right)<+\infty,\,\forall t>0\,,
\end{equation}
we define its \emph{Steiner rearrangement (in codimension $k$)} $\us:E^\sigma\to\R$ as
\begin{equation}\label{eq:steiner}
\us(x,y):=\inf\left\{t>0:\lambda_u(x,t)\leq\omega_k\abs{y}^k\right\}\,,
\end{equation}
where 
\[\lambda_u(x,t):=\Hk\left(\bigl\{y\in\Rk:u_0(x,y)>t\bigr\}\right)\]
is the \emph{distribution function (in codimension $k$)} of $u(x,\cdot)$ and $u_0$ is the extension of $u$ by $0$
outside $E$. Clearly,
$\us = 0$ in $\R^n\setminus E^\sigma$.
Let us observe that 
\begin{equation}\label{eq:schwarz}
\us(x,\cdot)=\bigl(u(x,\cdot)\bigr)^*\,, 
\end{equation}
where $\bigl(u(x,\cdot)\bigr)^*$ is the Schwarz rearrangement (which is also known as \textit{spherical symmetric
decreasing
rearrangement})
of $u$
with respect to the last $k$ variables.
Let us recall its definition. Given any non-negative measurable function $q:\Rk\to\R$, such that
$\Hn(\{y\in\Rk:u(y)>t\})$ is finite for all $t>0$, the \emph{Schwarz rearrangement} $q^*$ of $q$ is defined as
\[q^*(y):=\inf\{t>0:\mu(t)\leq \omega_k\abs{y}^k\}\,,\]
where $\mu(t):=\Hn\{y\in\Rk:q(y)>t\}$ is the \emph{distribution function} of $u$. The Schwarz
rearrangement satisfies an important property: it is non-expansive on $L^p(\R^k)$ for every $1\leq p< \infty$
(see, e.g., \cite{lieb-loss}*{Theorem 3.5}), i.e., for every $q_1$, $q_2\in L^p(\Rk)$
\[
\int_{\Rk}\abs{q_1^*-q_2^*}^p\leq \int_{\Rk}\abs{q_1-q_2}^p\,,
 \]
and this clearly implies the continuity of the Schwarz rearrangement on $L^p$.
Given any two non-negative measurable functions $u,v$ defined on $E$ and satisfying
\eqref{eq:misura_sezioni}, on applying the previous
inequality to $u^*(x,\cdot)$ and $v^*(x,\cdot)$ and integrating with respect to $x$, we see that
\begin{equation}\label{lemma:s_continuous}
 \norm{\us-v^\sigma}_{L^p(\Es)}\leq \norm{u-v}_{L^p(E)}\,,
 \end{equation}
 for all $1\leq p<+\infty$. In particular the Steiner rearrangement is continuous on $L^p$.

Given any non-negative and measurable function $u$, we define the \emph{subgraph} of $u$ as
\[\Su:=\left\{(x,y,t)\in\R^{n+1} : (x,y)\in E,\,0<t<u(x,y)\right\}.\]

Let us observe that for every $(x,t)\in\Rnk\times\Rt^+$, then $\Hk((\Su)_{x,t})=\lambda_u(x,t)$ and for
$\Hnk\aev x\in\Rnk$ we have $\us(x,y)>t$ if and only if $\lambda_u(x,t)>\omega_k\abs{y}^k$. Hence, we easily deduce that
\begin{equation}\label{eq:su_equiv_sus}
(\Su)^\sigma\text{\ and\ } \mathcal{S}_{\us}\text{\ are\ }\Ha^{n+1}\text{\ equivalent}.
\end{equation}
Moreover also the sets
$\{(x,y):u(x,y)>t\}^\sigma$ and $\{(x,y):\us(x,y)>t\}$ are equivalent (modulo $\Hn$) for every $t>0$. The latter fact
assures us that
$u$ and $\us$ are equidistributed functions. Actually, by the definition of the Steiner rearrangement, for $\Hnk\aev
x\in\proj(E)$ the functions $u(x,\cdot)$ and $\us(x,\cdot)$ are equidistributed. Therefore, Steiner rearrangement
preserves any so-called rearrangement invariant norm of a function, i.e., a norm depending only on the measure of its
level sets --- here important examples are any Lebesgue, Lorentz or Orlicz norm.

Let $f:\Rn\to [0,+\infty)$ be a non-negative convex function vanishing at $0$. We say that $f$ is radially symmetric
with respect to the
last $k$ variables if there exists a function $\ft:\R^{n-k+1}\to[0,+\infty)$ such that
\begin{equation}\label{eq:fconvex}
 f(x,y)=\ft(x,\abs{y})\,,
\end{equation}
for every $(x,y)\in\Rn$.

Given $f$ as above and an open set $\Omega$, we are interested in studying how functionals of the type
\begin{equation*}
u\mapsto \int_\Omega f(\nabla u)\,dz 
\end{equation*}
behave under Steiner rearrangement.
The class of admissible functions for these functionals will be
\[ \sobo(\Omega):=\left\{u:\Omega\to\R : u_0\in W^{1,1}(\omega\times\Ry),\,\forall \omega\Subset\proj(\Omega),
 \,\omega\text{\ open\ }\right\}.\]
Roughly speaking, $\sobo(\Omega)$ consists of those functions that are locally Sobolev with respect to the $x$ variable
and 
globally Sobolev with zero trace (in some appropriate sense) with respect to the $y$ variable. Let us remark that this
space is bigger than $W^{1,1}_0(\Omega)$. For instance, if $\Omega=[0,2\pi]^2$, the function $u=(\sin
y)/x\in\sobo(\Omega)$ but does not belong to $W^{1,1}_0(\Omega)$.
We can define, in a similar way, also the space $W^{1,p}_{0,y}(\Omega)$ for $p>1$.
For $\nabla u=(\partial_1 u,\dotsc,\partial_n u)$ we set 
\[\nabla_x u:=(\partial_1 u,\dotsc,\partial_{n-k}u)\text{\ and\ }
\nabla_y u:= (\partial_{n-k+1}u,\dotsc,\partial_n u),\]
where $\partial_i u:=\partial_{z_i}u(z)$ for $i=1,\dotsc,n$.

Note that the Steiner rearrangement maps $\sobo(\Omega)$ to $\sobo(\Omegas)$ (see
\cite{brock-s} and Proposition~ \ref{thm:us_is_sobolev} below). Let us
remark that in general the mapping is not continuous, see \cite{AL}.

We can now state the Pólya-Szegő principle for the Steiner rearrangement.
\begin{theorem}\label{thm:disuguaglianza}
 Let $f$ be a non-negative convex function, vanishing at $0$ and satisfying \eqref{eq:fconvex}. Let $\Omega\subset\Rn$
be an open set and $u\in\sobo(\Omega)$ be a non-negative function. Then
 \begin{equation}\label{eq:disuguaglianza}
 \int_{\Omegas}f(\nabla \us)\,dz\leq\int_\Omega f(\nabla u)\,dz\,.
\end{equation}
\end{theorem}

In Theorem \ref{thm:disuguaglianza} the space $\sobo(\Omega)$ can be replaced by any
space $W^{1,p}_{0,y}(\Omega)$, see Remark \ref{rk:w1p}.

We will call $u$ an \emph{extremal} if equality holds in \eqref{eq:disuguaglianza}. We are now interested to find
minimal assumptions to have a rigidity theorem for the extremals, i.e., in finding conditions that necessarily imply an
extremal
$u$ to be Steiner symmetric. It turns out that these assumptions concern both the function $u$ and the domain $\Omega$.

Regarding $u$, we set, for $x\in\proj(\Omega)$,
\[M(x):=\inf\{t>0:\lambda_u(x,t)=0\}\,.\]
Clearly, for $\Hnk\aev x\in\proj(\Omega)$,
\[M(x)=\essup \{u(x,y):y\in\Omega_x\}.\]
Also, $M$ is a measurable function in $\proj(\Omega)$ and by \eqref{eq:misura_sezioni} is finite for $\Hnk\aev
x\in\proj(\Omega)$. We require that
\begin{equation}\label{eq:condizioni_u}
 \Hn\bigl(\{(x,y)\in\Omega : \nabla_y u(x,y)=0\}
 \cap\{(x,y)\in\Omega:
 \text{either\ }M(x)=0 \text{\ or\ } u(x,y)<M(x)\}\bigr)=0\,.
\end{equation}
Roughly speaking, this condition means that the subgraph of $u$ does not contain any non trivial portion of a
$k$-dimensional hyperplane in the $y$-direction, except at the highest value of $u(x,\cdot)$.

\begin{remark}\label{rk:us}
It is known that the Schwarz rearrangement, in dimension $n\geq 2$, shrinks the set of critical points of
a Sobolev function (see \cite{AL}), while the Steiner rearrangement in codimension $1$ preserves its measure (see
\cite{burchard}). Hence,
by \eqref{eq:schwarz} and using the fact that the Steiner rearrangement of a Sobolev function is still weakly
differentiable (see Proposition \ref{thm:us_is_sobolev}), we have
\[\begin{split}
 \Hn\bigl(\{(x,y)\in\Omega:\nabla_y u(x,y)=0\}\bigr)&=
 \int_{\proj(\Omega)}\Hk\bigl(\{\nabla u(x,\cdot)=0\}\bigr)\,d\Hnk(x)\\
 &\leq \int_{\proj(\Omega)}\Hk\bigl(\{\nabla (u(x,\cdot))^*=0\}\bigr)\,d\Hnk(x)\\
 &=\Hn\bigl(\{(x,y)\in\Omegas:\nabla_y \us(x,y)=0\}\bigr)\,.
 \end{split}
\]
Therefore, if $u$ satisfies \eqref{eq:condizioni_u}
then the same holds for $\us$.
\end{remark}
Regarding the open set $\Omega$, we require that 
\begin{equation}
\label{eq:condizioni_om1}
\proj(\Omega)\text{\ is\ connected\ and\ }
 \Omega\text{\ is\ bounded\ in\ the\ }y\text{\ direction,}
\end{equation}
i.e., there exists $M>0$ such that $\Omega_x\subset B(0,M)$ for every $x\in\proj(\Omega)$, where $B(0,M)$ is the ball
in $\Rk$ of radius $M$ centered in $0$.
We also
require that, in some sense, the boundary of $\Omega$ is
almost nowhere parallel to the $y$-direction
inside the cylinder $\proj(\Omega)\times\Ry$. To be precise, we shall assume that 
\begin{equation}\begin{split}
\label{eq:condizioni_om3}
\Omega\text{\ is\ of\ finite\ perimeter\ inside\ }&\proj(\Omega)\times\Ry\,\text{\ and\ }\\
\Ha^{n-1}\bigl(\{(x,y)\in\brid\Omega:\nu^\Omega_y=0\}\cap\{&\proj(\Omega)\times\Ry\}\bigr)=0\,,
\end{split}\end{equation}
where $\brid\Omega$ stands for the reduced boundary of $\Omega$ and $\nu_y^\Omega$ is the $y$-component of the
generalized inner normal $\nu^\Omega$ of $\Omega$ --- see the next
section for the definitions.

We can now state the following result which gives a characterization of the equality cases in
\eqref{eq:disuguaglianza}.
\begin{theorem}\label{thm:equality}
 Let $f:\Rn\to\R$ be a non-negative strictly convex function satisfying \eqref{eq:fconvex} and vanishing in $0$.  Let
$\Omega\subset\Rn$ be
an open set satisfying \eqref{eq:condizioni_om1}$-$\eqref{eq:condizioni_om3} and let $u\in\sobo(\Omega)$ be a
non-negative
function. If
\begin{equation}\label{eq:equality}
 \int_{\Omegas} f(\nabla \us)\,dz=\int_\Omega f(\nabla u)\,dz<+\infty\,,
\end{equation}
then, for
$\Ha^{n-k+1}\aev 
 (x,t)\in\esspjt(\Su)$ there exists $R(x,t)>0$ such that the set
 \begin{equation*}
  \{y:u(x,y)>t\}\text{\ is\ equivalent\ to\ }\{\abs{y} < R(x,t)\}\,.
 \end{equation*}
If in addition $u$ satisfies \eqref{eq:condizioni_u}, then $\us$ is equivalent to $u$ up to a translation in the
$y$-plane.
\end{theorem}

At first sight, one could think that the assumptions made in the above statements are too strong. However, one can
easily construct counterexamples even in codimension $1$ (see \cite{CF}) showing that assumptions
\eqref{eq:condizioni_u}$-$\eqref{eq:condizioni_om3} cannot be weakened.

As we have seen before, if $u$ satisfies condition \eqref{eq:condizioni_u}, then the same
condition
holds for $\us$. In general the converse is not true, as one can see with some simple examples. However, it
turns out
that if equality holds in the Pólya-Szeg\H{o} inequality, then the two conditions are
equivalent.
\begin{proposition}\label{prop:nu_uguale}
 Let $f$ and $\Omega$ be as in Theorem \ref{thm:equality} and let $u\in\sobo(\Omega)$ be a non-negative function. If
equality \eqref{eq:equality} holds, then 
\begin{equation*}
 \Hn\bigl(\{(x,y)\in\Omega : \nabla_y u(x,y)=0\}\cap\{(x,y)\in\Omega : \text{either\ }
 M(x)=0 \text{\ or\ } u(x,y)<M(x)\}\bigr)=0
\end{equation*}
if and only if
\begin{equation}\label{eq:nu_ug_2}
 \Hn\bigl(\{(x,y)\in\Omegas : \nabla_y \us(x,y)=0\}\cap\{(x,y)\in\Omegas : \text{either\ }
 M(x)=0 \text{\ or\ } \us(x,y)<M(x)\}\bigr)=0\,.
\end{equation}
\end{proposition}

We now shift to the more general framework of functions of bounded variation. In this context, it is still possible to
show a Pólya-Szeg\H{o} principle, provided that the involved functional is properly defined. Consider any non-negative
convex function in $\Rn$ growing linearly at infinity, i.e.,  
for all $z\in\Rn$
\begin{equation}\label{eq:f-linear}
 0\leq f(z)\leq C(1+\abs{z})\,,
\end{equation}
for some positive constant $C$. Let us now define the \emph{recession function} $\frec$ of $f$ as
\[\frec(z):=\lim_{t\to +\infty}\frac{f(tz)}{t}\,.\]
Then a standard extension of the functional $\int_\Omega f(\nabla u)$ to the space $BV_\loc(\Omega)$ is defined as
\begin{equation}\label{eq:Jf}
 J_f(u;\Omega):=\int_\Omega f(\nabla u)\,dz+\int_\Omega \frec\left(\frac{\Ds u}{\abs{\Ds u}}\right)\, d\abs{\Ds u}\,.
\end{equation}

Here, $\nabla u$ stands for the approximate gradient of $u$, which agrees with the absolutely continuous part, with
respect to $\Hn$, of the measure $Du$, the distributional derivative of $u$. Also, $\Ds u$ is the singular part with
respect to $\Hn$ and $\abs{\Ds u}$ is its total variation. See the next section for the relevant definitions.
Actually, Theorem \ref{thm:frelax} states that $J_f(u;\Omega)$ coincides with the so-called \emph{relaxed functional} of
$\int_\Omega
f(\nabla u)$ in $BV(\Omega)$ with respect to the $L^1_\loc$-convergence.

Then, a Pólya-Szeg\H{o} principle for functionals of the form \eqref{eq:Jf} holds in the space of
$BV_\loc(\Omega)$ functions
vanishing in some appropriate sense on $\partial \Omega\cap(\proj(\Omega)\times\Ry)$.
To be precise, we set
\begin{multline*}
BV_{0,y}(\Omega):=
\Bigl\{u:\Omega\to\R \mid u_0\in BV(\omega\times\Ry)\text{\ and\ }
\abs{D u_0}(\omega\times\Ry)=\abs{Du}\bigl(\Omega\cap(\omega\times\Ry)\bigr)\\
\text{for\ every\ open\ set\ }\omega\Subset \proj(\Omega)\Bigr\}\,.
\end{multline*}

\begin{theorem}\label{thm:ineq-bv}
 Let $f:\Rn\to[0,+\infty)$ be a convex function vanishing at $0$ and satisfying \eqref{eq:fconvex} and
\eqref{eq:f-linear}. Let $\Omega\subset\Rn$ be an open set and let $u\in BV_{0,y}(\Omega)$ be a non-negative function.
Then $\us\in BV(\omega\times\Ry)$ for every open set $\omega\Subset\proj(\Omega)$ and
\begin{equation}\label{eq:ineq-bv}
 J_f(\us;\Omegas)\leq J_f(u;\Omega)\,.
\end{equation}
\end{theorem}

As before, we are interested in finding suitable conditions 
ensuring that a 
function satisfying the equality in \eqref{eq:ineq-bv} is Steiner
symmetric.
It turns out that one needs the same assumptions
on $u$ and $\Omega$ as in Theorem
\ref{thm:equality}.
Note that now the vector 
$\nabla_y u$ in \eqref{eq:condizioni_u} is the $y$-component of the absolutely continuous part of the measure $Du$.
However, in order to deal with the singular part 
$\Ds u$ of $Du$
we need some extra assumptions on the recession
function $\frec$. We will assume that for every $x\in\Rnk$, setting $\frec(x,y)=\ftrec(x,\abs{y})$,
\begin{equation}\label{eq:frec_incr}
 \ftrec(x,\cdot)\text{\ is\ strictly\ increasing\ on\ }[0,+\infty)
\end{equation}
and that the function
\begin{equation}\label{eq:frec_convex}
 x\mapsto\ftrec(x,1)\text{\ is\ strictly\ convex,}
\end{equation}

\begin{theorem}\label{thm:equal_bv}
 Let $f:\Rn\to [0,+\infty)$ be a strictly convex function vanishing at $0$ and satisfying
\eqref{eq:fconvex}, \eqref{eq:f-linear}, \eqref{eq:frec_incr} and \eqref{eq:frec_convex}. Let $\Omega\subset\Rn$ be an
open set satisfying \eqref{eq:condizioni_om1}$-$\eqref{eq:condizioni_om3} and let $u\in BV_{0,y}(\Omega)$ be a
non-negative function such that
\begin{equation}\label{eq:equal_bv}
 J_f(\us;\Omegas)=J_f(u;\Omega)<+\infty\,,
\end{equation}
Then, for $\Ha^{n-k+1}\aev (x,t)\in\esspj(\Su)$ there exists  $R(x,t)>0$ such
that the set
\[
 \{y:u(x,y)>t\}\text{\ is equivalent to }\{\abs{y}<R(x,t)\}\,.
\]
If in addition $u$ satisfies condition \eqref{eq:condizioni_u}, 
then $u$ is equivalent to $\us$ up to a translation in the $y$-plane.
\end{theorem}

The strategy in proving Theorems \ref{thm:ineq-bv} and \ref{thm:equal_bv} is to convert the functional $J_f$ into a
geometrical functional depending on the generalized inner normal and having the form
\begin{equation}\label{eq:Fnu}
 \int_{\brid E} F(\nu^E)\,d\Hn\,.
\end{equation}
Here, $F:\R^{n+1}\to [0,+\infty]$ is a convex function positively 1-homogeneous vanishing in $0$, i.e., for every
$\lambda>0$ and 
$(\xi_1,\dotsc,\xi_{n+1})\in\R^{n+1}$
\begin{equation}\label{eq:F-hom}
 F(\lambda\xi_1,\dotsc,\lambda\xi_{n+1})=\lambda F(\xi_1,\dotsc,\xi_{n+1})\quad \text{\ and\ }F(0)=0\,.
\end{equation}

Let us define 
\begin{equation}\label{eq:F_f}
  F_f(\xi_1,\dotsc,\xi_{n+1}):=
 \begin{cases}
f\left(-\frac{1}{\xi_{n+1}}(\xi_1,\dotsc,\xi_n)\right)(-\xi_{n+1}) & \text{\ if\ } \xi_{n+1}<0 \, , \\
\frec(\xi_1,\dotsc,\xi_n) & \text{\ if\ }\xi_{n+1}\geq 0\,.
\end{cases}
\end{equation}
The following result gives the link between the functional $J_f$ and the functional in \eqref{eq:Fnu}.
\begin{proposition}[\cite{CF}*{Proposition 2.7}]\label{prop:J=F}
 Let $f:\Rn\to [0,+\infty)$ be a convex function vanishing at $0$ and satisfying \eqref{eq:f-linear}. Then $F_f$ is a
convex function satisfying \eqref{eq:F-hom}. Moreover, if $\Omega\subset\Rn$ is an open set, then for every
non-negative function $u\in BV_\loc(\Omega)$ 
\begin{equation}\label{eq:J=F}
 J_f(u;\Omega)=\int_{\brid \Su\cap (\Omega\times\Rt)} F_f(\nu^{\Su})\,d\Hn\,.
\end{equation}
\end{proposition}

This allows us to reduce the proof of Theorem \ref{thm:ineq-bv} to the proof of a Pólya-Szeg\H{o} inequality for
functionals of the form \eqref{eq:Fnu}, where in addition we assume that $F$ is radial with respect to the $y$
variables,
i.e., there exists a function $\Ft:\R^{n-k+2}\to [0,+\infty]$ such that
\begin{equation}\label{eq:F-radial}
 F(x,y,t)=\Ft(x,\abs{y},t)\,,
\end{equation}
for every $(x,y,t)\in\R^{n+1}$.
Clearly, the function $\Ft$ is convex and positively 1-homogeneous.

It turns out that if $F$ satisfies \eqref{eq:F-hom} and \eqref{eq:F-radial} and if $E\subset\R^{n+1}$ is a set of
finite perimeter, then
\begin{equation}\label{eq:bv_in_aux}
 \int_{\brid E^\sigma} F(\nu^{E^\sigma})\,d\Hn\leq \int_{\brid E}F(\nu^E)\,d\Hn\,,
\end{equation}
see Theorem \ref{thm:F-ineq}. Then, Theorem \ref{thm:equal_bv} is proved thanks to Proposition \ref{prop:J=F} and to
a first characterization  of the equality cases in \eqref{eq:bv_in_aux} contained in Proposition \ref{thm:equal_F}.
In addition, an essentially complete characterization of the equality cases in \eqref{eq:bv_in_aux} is given by
Theorem  \ref{thm:SteinerF}.

Here, we want to point out that in order to give the characterization of the equality cases in \eqref{eq:equality} one
has to face with an extra difficulty. In fact, writing up
\[\begin{split}
 \lambda_u(x,t)&=\Ha^k\bigl(\{y\in\Rk : u_0(x,y)>t\}\cap\{\nabla_y u \neq 0\}\bigr)+
 \Ha^k\bigl(\{y\in\Rk : u_0(x,y)>t\}\cap\{\nabla_y u = 0\}\bigr)\\
  &=:\lambda^1_u(x,t)+\lambda_u^2(x,t)\,,
\end{split} \]
it turns out that $\lambda_u^1(x,t) \in W^{1,1}_\loc(\Rnk\times\Rt^+)$, while $\lambda^2$ is just a $BV$ function.
However, when $k=1$ the distributional derivative $D\lambda_u^2$ is purely singular with respect to the Lebesgue measure
on $\Rnk\times\Rt^+$, while if $k>1$ the measure $D\lambda_u^2$ may contain also a non-trivial absolutely continuous
part.  This fact was first observed in a celebrated paper by Almgren and Lieb \cite{AL} who
showed that this phenomenon may occur even if $u$ is a $C^1$ function.

\section{Background}\label{sec:background}
Given an open set $\Omega\subset\Rn$, we denote with $BV(\Omega)$ the class of functions of bounded variation, i.e.,
the family of functions in $L^1(\Omega)$ whose distributional gradient $Du$ is a vector-valued Radon measure in $\Omega$
of finite total variation $\abs{Du}(\Omega)$.
The space $BV_\loc(\Omega)$ is defined accordingly. 
By Lebesgue's Decomposition Theorem, the measure $Du$ can be split, with respect to the Lebesgue measure, in two
parts, the absolutely continuous part $\Da u$ and the singular part $\Ds u$. It turns out that $\Da u$ agrees
$\Hn$-a.e.
with $\nabla u$, the approximate gradient of $u$ (see, e.g., \cite{AFP}*{Definition 3.70}). Moreover, the set
$\mathcal{D}_u$ of all
points where $u$ is approximately differentiable satisfies $\abs{\Ds u}(\mathcal{D}_u)=0$ --- see, e.g., 
\cite{EG}*{\S 6.1, Theorem 4} or \cite{AFP}*{Theorem 3.83}.

A  measurable set $E\subset\Rn$ is said to be of \emph{finite perimeter} in an open set $\Omega\subset\Rn$ if
$D\chi_E$ is a vector-valued Radon measure with finite total variation in $\Omega$. The perimeter of $E$ in a Borel
subset $B$ of $\Omega$ is defined as $P(E;B):=\abs{D\chi_E}(B)$. For $B=\Rn$ we will simply write $P(E)$; if $\chi_E\in
BV_\loc(\Omega)$ then we say that $E$ has \emph{locally finite perimeter} in $\Omega$.

Denote by $u_x$ the function $u_x:\Omega_x\to\R$ defined by setting $u_x(y):=u(x,y)$ for all $x\in\proj(\Omega),\,
y\in\Omega_x$. From \cite{AFP}*{Theorems 3.103 and 3.107} we easily infer that for $\Hnk\aev x\in\proj(\Omega)$ the
function $u_x$ belongs to $BV(\Omega_x)$ and that
\begin{equation}\label{eq:derivata_sezione}
 \partial_{i} u_x(y)=\partial_{y_i} u(x,y),\,i=1,\dotsc,k\,,\text{\ for\ }\Hk\text{\aev}y\in\Omega_x\,.
\end{equation}

The following theorem (see \cite{GMS}*{\S 4.1.5, Theorem 1}) completely characterizes functions of bounded variation in
terms of their subgraphs. Let us remark that a slightly different notion of subgraph is needed here.
Given a function $u:\Omega\subset\Rn\to\R$, we set
\[\Su^-:=\{(x,y,t)\in\R^{n+1}:(x,y)\in\Omega,\,t<u(x,y)\}\,.\]
\begin{theorem}\label{thm:Su-}
 Let $\Omega\subset \Rn$ be a bounded open set and let $u\in L^1(\Omega)$. Then $\Su^-$ is a set of finite perimeter in
$\Omega\times\Rt$ if and only if $u\in BV(\Omega)$. Moreover, in this case,
 \[
  P(\Su^-;B\times\Rt)=\int_B\sqrt{1+\abs{\nabla u}^2}\,dz+\abs{\Ds u}(B)
 \]
for every Borel set $B\subset \Omega$.
\end{theorem}

Let $E$ be a set of finite perimeter in an open set $\Omega\subset\Rn$. For $i=1,\dotsc,n$ we denote by $\nu^E_i$ 
the derivative of the measure $D_i\chi_E$ with respect to $\abs{D\chi_E}$. Then, the \emph{reduced
boundary} $\brid E$ of $E$ consists of all points $ z $ of $\Omega$ such that the vector
$\nu^E( z ):=(\nu_1^E( z ),\dotsc,\nu_n^E( z ))$ exists and satisfies $\abs{\nu^E( z )}=1$. The vector $\nu^E( z )$ is
called the \emph{generalized inner normal} to $E$ at $z$. Moreover (see, e.g., {\cite{AFP}*{Theorem 3.59}), the
following formulae hold:
\begin{equation}\label{eq:DChi}\begin{split}
 D\chi_E&=\nu^E\Ha^{n-1}\rest\brid E\\
 \abs{D\chi_E}&=\Ha^{n-1}\rest \brid E\\
 \abs{D_i \chi_E}&=\abs{\nu^E_i}\Ha^{n-1}\rest\brid E\quad \text{for }i=1,\dotsc,n\,.
\end{split}\end{equation}

Given any measurable set $E\subset \Rn$, the \emph{density} of $E$ at $x$ is defined as
\[
 \Theta(E,x):=\lim_{r\to 0} \frac{\Hn\bigl(E\cap B(x,r)\bigr)}{\Hn\bigl(B(x,r)\bigr)}\,,
\]
provided that the limit on the right-hand side exists. Then, the \emph{measure theoretic boundary} of $E$ is the Borel
set defined as
\[
 \bmis E:=\Rn\setminus \{x\in\Rn : \text{either }\Theta(E,x)=0\text{ or }\Theta(E,x)=1\}\,.
\]
Given any two  measurable sets $E_1$ and $E_2$ in $\Rn$, we have
\begin{equation}\label{eq:measure_bdry}
 \bmis(E_1\cup E_2)\cup\bmis(E_1\cap E_2)\subset \bmis E_1\cup\bmis E_2\,.
\end{equation}
Moreover, if a set $E$ has locally finite perimeter in $\Omega$, the following holds (see, e.g., 
\cite{AFP}*{Theorem~ 3.61})
\begin{equation}\label{AZZ_(3.5)}
 \brid E\cap\Omega \subset \bmis E\cap \Omega\quad\text{\ and\ }\quad\Ha^{n-1}\bigl((\bmis E\setminus \brid
E)\cap\Omega\bigr)=0\,.
\end{equation}

The reduced boundary  of level sets plays an important role in the \emph{coarea formula} for functions of bounded
variations. In
its general version (see, e.g., \cite{AFP}*{Theorem 3.40}), it says that if $g:\Omega\to[0,+\infty]$ is any Borel
function and $u\in BV(\Omega)$, then
\begin{equation}\label{eq:gencoarea}
 \int_\Omega g\, d\abs{Du}=\int_{-\infty}^{+\infty}dt\int_{\Omega\cap\brid\{u>t\}}g\,d\Ha^{n-1}.
\end{equation}

The following proposition is a special case of the coarea formula for rectifiable sets (see \cite{AFP}*{Theorem 2.93})
\begin{proposition}
Let $\Omega\subset\Rn$ be an open set and let
E be a set of finite perimeter in $\Omega$. Let $g:\Omega\to[0,+\infty]$ be a Borel
function.
Then
 \begin{equation}\label{eq:coarea}
  \int_{\brid E \cap\Omega} g(z)\,\abs{\nu^\Omega_y(z)}\,d\Ha^{n-1}(z)=
  \int_{\proj(\Omega)} dx\int_{(\brid E\cap \Omega)_x} g(x,y)\,d\Ha^{k-1}(y).
 \end{equation}
\end{proposition}

Next theorem links the approximate gradient of a function of bounded variation to the generalized inner normal to its
subgraph --- see \cite{GMS}*{\S 4.1.5, Theorems 4 and 5}.
\begin{theorem}\label{thm:GMS2}
 Let $\Omega$ be an open subset of $\Rn$ and let $u\in BV(\Omega)$. Then
 \begin{equation}\label{eq:GMS2}
  \nu^{\Su^-}(x,y,t)=\left(
  \frac{\partial_1 u(x,y)}{\sqrt{1+\abs{\nabla u}^2}},\dotsc,
  \frac{\partial_{n} u(x,y)}{\sqrt{1+\abs{\nabla u}^2}},
  \frac{-1}{\sqrt{1+\abs{\nabla u}^2}}\right)
  \end{equation}
  for $\Hn\aev (x,y,t)\in\brid\Su^- \cap(\mathcal{D}_u\times\Rt)$ and
  \[\nu^{\Su^-}_{t}(x,y,t)=0\; \text{\ for\ }\Hn\aev (x,t)\in\brid\Su^-
\cap[(\Omega\setminus\mathcal{D}_u)\times\Rt].\]
In particular, if $u\in W^{1,1}(\Omega)$, then \eqref{eq:GMS2} holds for $\Hn\aev (x,t)\in\brid\Su^-
\cap(\Omega\times\Rt)$.
\end{theorem}

By Theorem \ref{thm:Su-}, if $\Omega$ is a bounded open set and $u\in BV(\Omega)$, the set $\Su^-$ has finite perimeter
in
$\Omega\times\Rt$. Thus, also $\Su$ has finite perimeter in $\Omega\times\Rt$; moreover 
\begin{equation}\label{eq:subgraph}\begin{split}
\brid\Su\cap(\Omega\times\Rt^+)=\brid\Su^-\cap(\Omega\times\Rt^+)\\
\nu^{\Su}\equiv\nu^{\Su^-}\text{\ on\ }\brid\Su\cap(\Omega\times\Rt^+).\quad\end{split}
\end{equation}

An important result we will use several times is Vol'pert's Theorem on
sections of sets of finite perimeter ---
see \cite{volpert} for the codimension $1$ case and \cite{barcafus}*{Theorem 2.4} for
the
general case.
\begin{theorem}\label{thm:volpert}
Let $E$ be a set of finite perimeter in $\Rn$. For $\Hnk\aev x\in\Rnk$ the following assertions hold:
\begin{enumerate} \renewcommand{\labelenumi}{\upshape(\roman{enumi})}
 \item $E_x$ has finite perimeter in $\Rk$;
 \item $\Ha^{k-1}((\brid(E_x)\triangle(\brid E)_x)=0$;
 \item For $\Ha^{k-1}\aev s$ such that $(x,s)\in \brid(E_x)$:
 \begin{enumerate}
  \item $\nu^E_y(x,s)\neq 0$;
  \item $\nu^E_y(x,s)=\nu^{E_x}(s)\abs{\nu^E_y(x,s)}$.
 \end{enumerate}\end{enumerate}
In particular, there exists a Borel set $G_E\subset\esspj(E)$ such that $\Hnk(\esspj(E)\setminus
G_E)=0$
and \upshape{(i)--(iii)} hold for every $x\in G_E$.
\end{theorem}
In view of the previous theorem, we will use the same notation $\brid E_x$ to denote $(\brid E)_x$ and $\brid (E_x)$
when they coincide up to $\Ha^{k-1}$ negligible sets. 

Next result, proved in \cite{barcafus}*{Lemma 3.1}, deals with some properties of the function $L$ and its derivatives.
Recall from Section \ref{sec:states} that $L(x):=\Ha^k(E_x)$.
\begin{lemma}
Let $E$ be any set of finite perimeter in $\Rn$. Then, either $L(x)=+\infty$ for $\Hnk\aev
x\in\Rnk$ or $L(x)<+\infty$ for $\Hnk\aev x\in\Rnk$ and $\Hn(E)<+\infty$. Moreover, in the latter case, $L\in BV(\Rnk)$
and for any Borel set $B\subset\Rnk$
\begin{multline}\label{eq:formula_deriv}
 DL(B)=\int_{\brid E\cap (B\times\Rk)\cap\{\nu^E_y=0\}} \nu^E_x(x,y)\,d\Ha^{n-1}(x,y)\\
  + \int_B dx\int_{\brid E_x\cap \{\nu_y^E\neq 0\}}
\frac{\nu^E_x(x,y)}{\abs{\nu^E_y(x,y)}}\,d\Ha^{k-1}(y)\,,
\end{multline}
$DL\rest G_{\Es}=\nabla L\,\Hnk$ and for $\Hnk\aev x\in G_{\Es}$
\begin{equation}\label{eq:lemma_partialLEs}
 \nabla L(x)=\Ha^{k-1}(\brid \Es_x)\frac{\nu^{\Es}_x(x)}{\abs{\nu^{\Es}_y(x)}}\,,
\end{equation}
where we dropped the variable $y$ for functions that are constant in $\brid \Es_x$.
\end{lemma}

\section{The Sobolev case}\label{sec:sobolev}
In this section we prove the Pólya-Szegő inequality for the Steiner rearrangement in codimension $k$ of
Sobolev functions and Theorem \ref{thm:equality} concerning the equality cases. 

We first observe that
the Steiner rearrangement of a function in $\sobo(\Omega)$ belongs to $\sobo(\Omegas)$.
\begin{proposition}\label{thm:us_is_sobolev}
 Let $\Omega\subset\Rn$ be an open set and let $u\in\sobo(\Omega)$ be a non-negative function.
Then $\us\in\sobo(\Omegas)$.
\end{proposition}
\begin{proof}
By \cite{brock-s}*{Theorem 8.2} we know that if $v\in W^{1,1}(\Omega)$ is a non-negative function, then $v^\sigma$
belongs to $W^{1,1}(\Omegas)$. Given a non-negative function $u\in\sobo(\Omega)$ and 
fixed $\omega\Subset\proj(\Omega)$ we can find a cut-off function $\phi\in C^1_c(\proj(\Omega))$ such that
$\phi\equiv 1$ in $\omega$. Hence, the function $v:=\phi u$ belongs to $W^{1,1}(\Omega)$. Then, $v^\sigma\in
W^{1,1}(\Omegas)$. Besides, $v^\sigma(x,y)=\us(x,y)$ for all $x\in\omega$ and $y\in\Rk$. This proves the assertion.
\end{proof}

Next lemma gives
formulae for the approximate derivatives of the distribution function of a Sobolev function.
\begin{lemma}\label{lemma:derivate}
 Let $\Omega\subset\R^n$ be an open and bounded set, $u:\Omega\to\R$  be a non-negative function,
 $u\in \sobo(\Omega)$ satisfying \eqref{eq:condizioni_u}. Then, $\lambda_u\in W^{1,1}(\omega\times\Rt^+)$ for every open
set
$\omega\Subset\proj(\Omega)$
and for $\Hnk\aev x\in\esspj(\Su)$,
\begin{equation}\label{eq:tesi_deriv1}
 \partial_t\lambda_u(x,t)=-\int_{\brid\{y:u(x,y)>t\}}\frac{1}{\abs{\nabla_y u}}\;d\Ha^{k-1}(y),
\end{equation}
\begin{equation}\label{eq:tesi_deriv2}
 \partial_i\lambda_u(x,t)=\int_{\brid\{y:u(x,y)>t\}}\frac{\partial_i u}{\abs{\nabla_y u}}\;d\Ha^{k-1}(y)\,
,\,i=1,\dotsc,n-k,
\end{equation}
for $\Ha^1\aev t\in(0,M(x))$.
\end{lemma}
\begin{proof}
 Let $r>0$ be large enough to have $\Omega\subset\Rnk\times B(0,r)$ and let $\omega\Subset\proj(\Omega)$ For the sake of
simplicity we shall identify the extension $u_0$ with $u$. Hence, we may
assume that $u\in W^{1,1}(\omega\times\Ry)$ and $u(x,y)=0$ if $\abs{y}>r$.

If $\phi\in C^1_c(\Omega\times\Rt^+)$, by Fubini's Theorem we get, for $i=1\dotsc,n-k$,
\begin{equation}\label{eq:deriv_1}\begin{split}
&\int_{\omega\times\Rt^+}\partial_i \phi(x,t)\lambda_u(x,t)\,dx\, dt 
 =\int_{\omega\times\Ry\times\Rt^+} \partial_i\phi(x,t)\chi_{\Su}(x,y,t)\,dx\,dy\,dt\\
 &\quad=\int_{\omega\times\Ry} dx\,dy\int_0^{u(x,y)} \partial_i \phi(x,t)\,dt
 \\
 &\quad=\int_{\omega\times B(0,r)}\partial_i\left[ \int_0^{u(x,y)} \phi(x,t)\,dt\right]dx\,dy
  -\int_{\omega\times B(0,r)} \phi(x,u(x,y))\partial_i u(x,y)\,dx\,dy
 \end{split}\end{equation}
The first integral in the last expression vanishes over $\omega\times B(0,r)$. Applying the coarea formula
\eqref{eq:coarea} and recalling that by Theorem \ref{thm:volpert} 
\[(\brid \Su)_{x,y}\cap\Rt^+=\brid(\Su)_{x,y}\cap\Rt^+=\brid(0,u(x,y))\cap\Rt^+\] for $\Hn\aev (x,y)\in\omega\times
B(0,r)$, we get
\begin{equation}\label{eq:deriv_2}
\begin{split}
& \int_{\brid\Su\cap(\omega\times B(0,r)\times\Rt^+)} \phi(x,t)\partial_i u(x,y)\abs{\nu_t^{\Su}(x,y,t)}\,d\Hn \\
&\quad=  \int_{\omega\times B(0,r)}dx\,dy\int_{(\brid\Su)_{x,y}\cap\Rt^+} \phi(x,t)\partial_i u(x,y)\,d\Ha^0(t)\\
&\quad= \int_{\omega\times B(0,r)} \phi(x,u(x,y))\partial_i u(x,y)\,dx\,dy\,.
\end{split}
\end{equation}
Moreover, from \eqref{eq:GMS2} and \eqref{eq:subgraph}, we have
\begin{equation}\label{eq:deriv_3}
 \nu^{\Su}(x,y,t)=\left(\frac{\nabla_x u(x,y)}{\sqrt{1+\abs{\nabla u}^2}},
 \frac{\nabla_y u(x,y)}{\sqrt{1+\abs{\nabla u}^2}}, \frac{-1}{\sqrt{1+\abs{\nabla u}^2}}
 \right)
\end{equation}
for $\Hn\aev (x,y,t)\in\brid\Su\cap (\omega\times B(0,r)\times\Rt^+)$.

Combining \eqref{eq:deriv_1}$-$\eqref{eq:deriv_3}, we have
\begin{equation}\begin{split}\label{eq:deriv_4}
   \int_{\omega\times\Rt^+}\partial_i \phi(x,t)\lambda_u(x,t)\,dx\, dt
 &=-\int_{\brid\Su\cap(\omega\times B(0,r)\times\Rt^+)} \phi(x,t)\partial_i u(x,y)\abs{\nu_t^{\Su}(x,y,t)}\,d\Hn\\
 &=-\int_{\brid\Su\cap(\omega\times B(0,r)\times\Rt^+)} \phi(x,t)\partial_i u(x,y)\cdot \frac{1}{\sqrt{1+\abs{\nabla
u}^2}} d\Hn.
\end{split}\end{equation}
The last equation implies that the distributional derivative $D_i \lambda_u$ is a finite Radon measure on
$\omega\times\Rt^+$. A similar argument shows that the same holds for  $D_t \lambda_u$. Therefore, since
\[\int_{\omega\times\Rt^+}\lambda_u (x,t)dx\,dt=\int_{\omega\times\Ry}u(x,y)\,dx\,dy<+\infty\,,\] we get 
$\lambda_u\in L^1(\omega\times\Rt^+)$ and thus $\lambda_u\in BV(\omega\times\Rt^+)$. 

Notice that \eqref{eq:deriv_4} implies that for every $\phi\in C^1_c(\omega\times\Rt^+)$ we have
\begin{equation}\label{eq:deriv_5}
 \int_{\omega\times\Rt^+} \phi(x,t)\,dD_i\lambda_u =\int_{\brid\Su\cap(\omega\times B(0,r)\times\Rt^+)}
 \phi(x,t)\cdot \frac{\partial_i u(x,y)}{\sqrt{1+\abs{\nabla u}^2}}d\Hn\,.
\end{equation}
By density, the same equality holds for $\phi\in C(\omega\times\Rt^+)$.

We claim that \eqref{eq:deriv_5} holds also for every bounded Borel function in $\omega\times\Rt^+$.
In fact, for any Borel set $B\subset\omega\times\Rt^+$, define the Borel measure $\mu$ by setting
\[\mu(B):=\abs{D_i\lambda_u}(B)+\Hn\left(\brid\Su\cap (B\times\Ry)\right)\]
and let $\phi$ be any bounded Borel function in $\omega\times\Rt^+$. By Lusin's Theorem, for any $\epsilon>0$ there
exists a function $\phi_\epsilon\in C(\omega\times\Rt^+)$ such that ${\norm{\phi_\epsilon}}_\infty \leq\
{\norm{\phi}}_\infty$ and $\mu\{(x,t) :\phi_\epsilon(x,t)\neq\phi(x,t)\}<\epsilon$. 
Since $\phi_\epsilon$ is continuous, equality \eqref{eq:deriv_5} holds for
$\phi_\epsilon$, 
and hence the absolute value
of the difference of the left-hand
side and the right-hand side is not greater than $4\epsilon{\norm{\phi}}_\infty$. From
the arbitrariness  of $\epsilon$, the claim follows.

Let $g\in C_c(\omega\times\Rt^+)$. From \eqref{eq:deriv_5}, \eqref{eq:deriv_3} and using condition
\eqref{eq:condizioni_u} with the coarea formula \eqref{eq:coarea}, we get
\begin{equation*}\begin{split}
\int_{\omega\times\Rt^+} g(x,t)\,d D_i\lambda_u
&=\int_{\brid\Su\cap(\omega\times\Ry\times\Rt^+)} g(x,t)\partial_i
u(x,y)\cdot\frac{1}{\sqrt{1+\abs{\nabla u}^2}} d\Hn\\
&=\int_{\brid\Su\cap(\omega\times\Ry\times\Rt^+)} g(x,t) \frac{\partial_i u(x,y)}{\abs{\nabla_y
u(x,y)}}
\abs{\nu_y^{\Su}(x,y,t)}\,d\Hn\\
&=\int_{\omega\times\Rt^+}g(x,t)\, dx\,dt
\int_{(\brid\Su)_{x,t}}\frac{\partial_i u(x,y)}{\abs{\nabla_y u(x,y)}} d\Ha^{k-1}(y)\,.
\end{split}\end{equation*}
Since $g$ is arbitrary, we have that the measure $D_i\lambda_u$ is absolutely continuous with respect to
$\Ha^{n-k+1}$
and is equal to
\[\left(\int_{(\brid\Su)_{x,t}}\frac{\partial_i u(x,y)}{\abs{\nabla_y u(x,y)}}
d\Ha^{k-1}(y)\right)\Ha^{n-k+1}\,,\]
thus proving that $\lambda_u\in W^{1,1}(\omega\times\Rt^+)$.
Because of \upshape{(ii)} in Theorem \ref{thm:volpert}, equation \eqref{eq:tesi_deriv2} holds for
$\Ha^{n-k+1}\aev (x,t)\in\esspjt(\Su)\cap(\omega\times\Rt^+).$

Since 
\begin{equation}\label{eq:equivalenti}
\esspjt(\Su)\text{\ is\ equivalent\ to\ }\bigcup_{x\in\esspj(\Su)} \{x\}\times(0,M(x))\,,
\end{equation}
we see that for
$\Hnk\aev x\in\esspj(\Su)$ equation \eqref{eq:tesi_deriv2} holds for $\Ha^1\aev t\in (0,M(x))$.

It remains to prove \eqref{eq:tesi_deriv1}: this follows from the same calculations and applying
\eqref{eq:derivata_sezione} and \eqref{eq:formula_deriv}.
\end{proof}

\begin{remark}
 If $\Omega$ and $u$ are as in Lemma \ref{lemma:derivate}, then, by Proposition \ref{thm:us_is_sobolev} 
$\us\in\sobo(\Omega)$, by Remark \ref{rk:us} $\us$ satisfies condition \eqref{eq:condizioni_u} and we get
that for $\Hnk\aev x\in\esspj(\Su)$
\begin{gather}
 \label{eq:nabla_t_u}\partial_t\lambda_u(x,t)=-\frac{\Ha^{k-1}(\brid\{y:\us(x,y)>t\})}{\abs{\nabla_y \us}}|_
 {\brid\{y:\us(x,y)>t\}}\\
 \label{eq:nabla_i_u}\partial_i\lambda_u(x,t)=\Ha^{k-1}(\brid\{y:\us(x,y)>t\}) \frac{\partial_i \us}{\abs{\nabla_y
\us}}|_
 {\brid\{y:\us(x,y)>t\}}
\end{gather}
\end{remark}

The following approximation result will be useful in the proof of Theorem \ref{thm:disuguaglianza}.
\begin{lemma}\label{lemma:approx}
 Let $\omega\subset\Rnk$ be an open set and let $u\in W^{1,p}(\omega\times\Ry)$, $p\geq 1$, be a non-negative function.
Then for
every $\omega'\Subset\omega$ and
for every $\epsilon>0$ there exists a non-negative Lipschitz function $w:\Rn\to\R$ with compact support such that
\begin{gather}
 \label{eq:lemma_approx_1}
 \Hn\left(\{z\in\Rn:w(z)>0,\,\nabla_y w(z)=0\}\right)=0\text{\ and\ }\\
 \label{eq:lemma_approx_2}
 {\norm{u-w}}_{W^{1,p}(\omega'\times\Ry)} < \epsilon\,.
\end{gather}
\end{lemma}
\begin{proof}
On multiplying $u(x,y)$ by a smooth compactly supported cut-off function $\phi:\Rnk\to\R$ with $\phi\equiv 1$ on
$\omega'$, we can assume without loss of generality that $u\in W^{1,p}(\Rn)$. By density, for every choice of
$\epsilon>0$ there exists a non-negative function $u_\epsilon\in C^1_c(\Rn)$ such that
${\norm{u-u_\epsilon}}_{W^{1,p}(\Rn)}<\epsilon$.

Let $r>1$ be such that $\supp u_\epsilon\subset B(0,r)$. Standard approximation results assure us that there exists a
polynomial  $p_\epsilon$ such that ${\norm{u_\epsilon-p_\epsilon}}_{C^1(\bar{B}(0,2r)} < \epsilon/r^{n/p}$. On
replacing, if necessary, $p_\epsilon$ with $p_\epsilon+\epsilon/r^{n/p}+\delta\abs{y}^2$, for $\delta >0$ sufficiently
small, we may assume $p_\epsilon$ to be strictly positive and $\nabla_y p_\epsilon\neq 0\, \Hn\aev$on
$\bar{B}(0,r)$.

Define $\eta_r:\Rn\to\R$ as
\[\eta_r(z)=\begin{cases}
1 &\text{if\ }\abs{z}\leq r\\
\frac{(4r^2-\abs{z}^2)}{3r^2} &\text{if\ } r<\abs{z}\leq 2r\\
0 &\text{if\ }\abs{z}> 2r
            \end{cases}\]
and let $w=p_\epsilon\eta_r$. Then there exists a constant $c=c(n,p)>0$ such that 
${\norm{u-w}}_{W^{1,p}(\Rn)} <c\epsilon$ and so equation \eqref{eq:lemma_approx_2} holds.

Finally, \eqref{eq:lemma_approx_1} is proven by considering that $w(z)>0$ if and only if $z\in B(0,2r)$ and that
$w\equiv p_\epsilon$ on $B(0,r)$ and $w\equiv p_\epsilon\eta_r$ on $B(0,2r)\setminus\bar{B}(0,r)$ and hence $w$ is still
a polynomial with $\nabla_y w\neq 0\,\Hn\aev$
\end{proof}

\begin{proof}[Proof of Theorem \ref{thm:disuguaglianza}.]
We are going to prove a stronger inequality that
actually
implies \eqref{eq:disuguaglianza}, i.e., 
 \begin{equation}\label{eq:disug-diversa}
  \int_{B\times\Ry} f(\nabla \us)\, dz\leq\int_{B\times\Ry}f(\nabla u)\,dz\,,
 \end{equation}
for every Borel set $B\subset\proj(\Omega)$. As before, we will identify $u$ with its extension $u_0$.
We can assume that the right-hand side of \eqref{eq:disug-diversa} has finite value. If not the inequality trivially
holds.

\step{1}
Let us first prove inequality \eqref{eq:disug-diversa} under additional assumptions: we assume that $\Omega$ is bounded
with
respect
to the last $k$ components and that $u\in\sobo(\Omega)$ is
non-negative and satisfies
\begin{equation}\label{eq:pf_dis1}
 \Hk \bigl( \{y\in\Rk:\nabla_y u(x,y)=0\}\cap\{y\in\Rk:u(x,y)>0\}\bigr)=0
\end{equation}
for $\Hnk\aev x\in\proj(\Omega)$. By Remark \ref{rk:us}, equation \eqref{eq:pf_dis1} holds
also for $\us$. On applying the coarea formula \eqref{eq:gencoarea} and \eqref{eq:derivata_sezione}, we get that 
\begin{equation}\label{eq:pf_dis1bis}
 \int_{\{y:\us(x,y)>0\}}f(\nabla\us)\,dy=
 \int_0^{+\infty}dt\int_{\brid\{y:\us(x,y)>t\}}\frac{f(\nabla\us)}{\abs{\nabla_y\us}}d\Ha^{k-1}\,,
\end{equation}
for
$\Hnk\aev x\in\proj(\Omega)$.
Hence, for any such $x$, assumption \eqref{eq:fconvex} and \eqref{eq:nabla_t_u}$-$\eqref{eq:nabla_i_u} give
\begin{equation}\begin{split}\label{eq:pf_dis3}
&\int_{\brid\{y:\us(x,y)>t\}}\frac{1}{\abs{\nabla_y\us}}f(\partial_1\us,\dotsc,\partial_{n-k}\us,\dotsc,\partial_n\us)\,
d\Ha^{k-1}\\
 &\quad=\int_{\brid\{y:\us(x,y)>t\}}\frac{1}{\abs{\nabla_y\us}}\ft(\partial_1\us,\dotsc,\partial_{n-k}\us, 
 \abs{\nabla_y\us})\,  d\Ha^{k-1}\\
 &\quad= -\partial_t\lambda_u(x,t)\ft\left(\frac{\nabla_x\lambda_u(x,t)}{-\partial_t\lambda_u(x,t)},
 \frac{\Ha^{k-1}(\brid\{y:\us(x,y)>t\})}{-\partial_t\lambda_u(x,t)}\right)\,,\end{split}
\end{equation}
for $\Ha^1\aev t>0$.
Let us note that for $\Hnk\aev x\in\proj(\Omega)$, the set $\{y:u(x,y)>t\}\subset\Rk$ is of finite perimeter for
$\Ha^1\aev t>0$ and $\Hk(\{y:u(x,y)>t\})<+\infty$ for $t>0$. By the isoperimetric inequality in $\Rk$, 
\begin{equation}\label{eq:pf_dis2}
 \Ha^{k-1}(\brid\{y:\us(x,y)>t\})\leq\Ha^{k-1}(\brid\{y:u(x,y)>t\})
 =\int_{\brid\{y:u(x,y)>t\}}d\Ha^{k-1}
\end{equation}
holds for $\Hnk\aev x\in\proj(\Omega)$, for $\Ha^1\aev t>0$. By assumption
\eqref{eq:fconvex} the function $\ft(\xi,\cdot)$ is non decreasing in $[0,+\infty)$ for every $\xi\in\Rnk$.
Therefore, \eqref{eq:pf_dis2} and Lemma \ref{lemma:derivate} imply that for $\Hnk\aev x\in\proj(\Omega)$
\begin{multline}\label{eq:pf_dis4}
 -\partial_t\lambda_u(x,t)\ft\left(\frac{\nabla_x\lambda_u(x,t)}{-\partial_t\lambda_u(x,t)},
 \frac{\Ha^{k-1}(\brid\{y:\us(x,y)>t\})}{-\partial_t\lambda_u(x,t)}\right)\\
 \leq\ft\left(\frac{\int_D\frac{\partial_1 u}{\abs{\nabla_y u}}d\Ha^{k-1}}
 {\int_D\frac{1}{\abs{\nabla_yu}}d\Ha^{k-1}},\dotsc,
 \frac{\int_D\frac{\partial_{n-k} u}{\abs{\nabla_y u}}d\Ha^{k-1}}
 {\int_D\frac{1}{\abs{\nabla_yu}}d\Ha^{k-1}},\frac{\int_D d\Ha^{k-1}}{\int_D\frac{d\Ha^{k-1}}{\abs{\nabla_y u}}}
 \right)\cdot
 \int_D\frac{d\Ha^{k-1}}{\abs{\nabla_y u}}=:\mathcal{I}
\end{multline}
for $\Ha^1\aev t>0$, where $D:=\brid\{y:u(x,y)>t\}$. Recalling that $f$ is convex and so $\ft$ is, Jensen's
inequality gives
\begin{equation}\label{eq:pf_dis5}
 \mathcal{I}\leq\int_{\brid\{y:u(x,y)>t\}}\frac{1}{\abs{\nabla_y u}}\ft(\nabla_x u,\abs{\nabla_y u})\, d\Ha^{k-1}.
\end{equation}
Putting together \eqref{eq:pf_dis3}, \eqref{eq:pf_dis4} and \eqref{eq:pf_dis5} we get
\begin{multline}\label{eq:pf_dis6}
 \int_{\brid\{y:\us(x,y)>t\}} \frac{1}{\abs{\nabla_y \us}}\ft(\nabla_x\us,\abs{\nabla_y \us})\,d\Ha^{k-1}\\
 \leq \int_{\brid\{y:u(x,y)>t\}} \frac{1}{\abs{\nabla_y u}}\ft(\nabla_x u,\abs{\nabla_y u})\,d\Ha^{k-1}\,,
\end{multline}
for $\Hnk\aev x\in\proj(\Omega)$ and for $\Ha^1\aev t>0$.

Integrating \eqref{eq:pf_dis6}, first with respect to $t$ and
then with respect to $x$, using equation \eqref{eq:pf_dis1bis} for both $u$ and $\us$, yields
\begin{equation}\label{eq:pf_dis7}\begin{split}
 \int_{B\times\Ry} f(\nabla \us) \, dx \, dy &= \int_{B} dx \int_{\brid\{y:\us(x,y)>0\}} f(\nabla\us)\,dy \\
 &=\int_{B} dx \int_0^{+\infty} dt \int_{\brid\{y:\us(x,y)>t\}}\frac{f(\nabla\us)}{\abs{\nabla_y \us}}d\Ha^{k-1}\\
 &\leq \int_{B} dx\int_0^{+\infty}dt \int_{\brid\{y:u(x,y)>t\}}\frac{f(\nabla u)}{\abs{\nabla_y u}}d\Ha^{k-1}\\
 &=\int_{B\times\Ry} f(\nabla u)\,dx\,dy\,.\end{split}
\end{equation}

\step{2} Let us remove the additional assumptions we used in Step 1. Let $u\in\sobo(\Omega)$ be non-negative and let
$\omega\Subset\proj(\Omega)$ be an open set. Lemma
\ref{lemma:approx}
gives the existence of a sequence $\{u_h\}$ of non-negative Lipschitz functions, compactly supported in $\Rn$, that
satisfy
\eqref{eq:pf_dis1} and such that $u_h\to u$ strongly in $W^{1,1}(\omega\times\Ry)$.

If we assume that 
\begin{equation}\label{eq:pf_dis_linear}
 0\leq f(\xi)\leq C(1+\abs{\xi})\,\text{\ for\ some\ }C>0,\quad\forall \xi\in\Rn,
\end{equation}
then $f$ is globally Lipschitz continuous and therefore $f(\nabla u_h)\to f(\nabla u)$ strongly in
$L^1(\omega\times\Ry)$. The continuity of Steiner symmetrization, see equation \eqref{lemma:s_continuous}, with respect
to
the
$L^1$-convergence gives us $\us_h\to\us$ strongly in $L^1(\omega\times\Ry)$. By semicontinuity 
(see, e.g., \cite{buttazzo}*{Theorem 4.2.8}) and \eqref{eq:pf_dis7} we have
\[\begin{split}
\int_{\omega\times\Ry}f(\nabla\us)\,dx\,dy&\leq \liminf_{h\to +\infty} \int_{\omega\times\Ry} f(\nabla\us_h)\,dx\,dy\\
&\leq\liminf_{h\to +\infty}\int_{\omega\times\Ry} f(\nabla u_h)\,dx\,dy= \int_{\omega\times\Ry} f(\nabla u)\,dx\,dy\,,
  \end{split}\]
and so \eqref{eq:disug-diversa} holds.

Let us remove assumption \eqref{eq:pf_dis_linear}. Since $f$ is non-negative and convex and satisfies
\eqref{eq:fconvex}, there exist a sequence of
vectors $\{a_j\}\subset\Rnk$ and two sequences of numbers $\{b_j\}\subset\R$, $\{c_j\}\subset\R$ such that
\[f(\xi)=\sup_{j\in\N}\{a_j\cdot\xi_x +b_j\abs{\xi_y}+c_j\}=\sup_{j\in\N}\{(a_j\cdot\xi_x
+b_j\abs{\xi_y}+c_j)^+\}\,,\quad\forall\xi\in\Rn\,.\]
For $N\in\N$ define
\begin{equation*}
 f_N(\xi):=\sup_{1\leq j\leq N}
 \{(a_j\cdot\xi_x +b_j\abs{\xi_y}+c_j)^+\}\,.
\end{equation*}
Clearly, $f_N(\xi)\nearrow f(\xi)$ pointwise monotonically. Observing that $f_N$ satisfies \eqref{eq:fconvex} and
\eqref{eq:pf_dis_linear} we get that \eqref{eq:disug-diversa} holds for such $f_N$. Now the thesis follows by monotone
convergence theorem.
\end{proof}

\begin{remark}\label{rk:w1p}
Actually, inequality \eqref{eq:disuguaglianza} holds also for any $u$ in $W^{1,p}_{0,y}(\Omega)$. To verify this,
define, for
every
$\epsilon>0$, $u_\epsilon:=\max\{u-\epsilon,0\}$. Clearly, the support of $u_\epsilon$ has finite measure in
$\omega\times\Ry$
for every $\omega\Subset\proj(\Omega)$. Therefore $u_\epsilon\in \sobo(\Omega)$. Since
$(u_\epsilon)^\sigma=(\us)_\epsilon$ and $\nabla u_\epsilon=\nabla u \chi_{\{u>\epsilon\}}\, \Hn\aev$in $\Rn$, by
monotone convergence theorem and applying \eqref{eq:disug-diversa} to $u_\epsilon$, we get
\[\begin{split}
 \int_{B\times\Ry}f(\nabla\us)\,dz&=\lim_{\epsilon\to 0^+}\int_{B\times\Ry}f(\nabla(\us)_\epsilon)\,dz
 =\lim_{\epsilon\to 0^+}\int_{B\times\Ry}f(\nabla (u_\epsilon)^\sigma)\,dz\\
 &\leq \lim_{\epsilon\to 0^+}\int_{B\times\Ry}f(\nabla u_\epsilon)\,dz
 =\int_{B\times\Ry}f(\nabla u)\, dz\,.
 \end{split}
\]
\end{remark}

We now pass to the equality cases. Next result shows that if equality holds in the Pólya-Szegő inequality, then almost
every $(x,t)$-section of the subgraph is equivalent to a ball.
\begin{lemma}\label{lemma:sezioni_palle}
 Let $f:\Rn\to\R$ be a non-negative strictly convex function satisfying \eqref{eq:fconvex} that vanishes in $0$ 
 and let $u\in\sobo(\Omega)$ be a non-negative function.
If equality \eqref{eq:equality} holds, then for
$\Ha^{n-k+1}\aev 
 (x,t)\in\esspjt(\Su)$ there exists  $R(x,t)>0$ such that the set
 \begin{equation*}
  \{y:u(x,y)>t\}\text{\ is\ equivalent\ to\ }\{\abs{y} < R(x,t)\}.
 \end{equation*}
 \end{lemma}
\begin{proof}
We prove here the lemma under the additional assumption that $u$ satisfies \eqref{eq:condizioni_u}. For the general
case see Remark \ref{rk:fsobo}.

Assumption \eqref{eq:equality} and inequality \eqref{eq:disug-diversa} imply that 
\begin{equation}\label{eq:pr_palla}
 \int_{B\times\Ry} f(\nabla \us)\,dz=\int_{B\times\Ry} f(\nabla u)\,dz
\end{equation}
for every Borel set $B\subset \proj(\Omega)$. On choosing $A:=\esspj(\Omega)\cap G_{\Su}\cap G_{{\Sus}}$, from Theorem
\ref{thm:volpert} and \eqref{eq:GMS2} we see that $\Ha^{n-k}(\esspj(\Omega)\setminus A)=0$ and that $\nabla_y
u(x,y)\neq 0$ on $A\times\Ry$. 

Equality \eqref{eq:pr_palla} assures us that equality holds in \eqref{eq:pf_dis7} with $B$ replaced by $A$. By
\eqref{eq:condizioni_u} $u$ is $\Hn\aev$strictly positive in $\Omega$, and therefore we have equalities also in
\eqref{eq:pf_dis4} and \eqref{eq:pf_dis5}. Since $\ft(\xi,\cdot)$ is strictly increasing in $[0,+\infty)$ we get an
equality in \eqref{eq:pf_dis2}. Applying the isoperimetric theorem in $\Rk$, is clear that $\{y:u(x,y)>t\}$ is
equivalent to a ball of radius $R(x,t)$ for $\Hnk\aev x\in\proj(\Omega)$ and $\Ha^1\aev t\in(0,M(x))$. By the 
$\Hn\aev$positivity of $u$, we have that $\esspj(\Su)$ is equivalent to $\proj(\Omega)$. Equation
\eqref{eq:equivalenti} implies that $\esspjt(\Su)$ is equivalent to $\bigcup_{x\in\proj(\Omega)}
\{x\}\times(0,M(x))$. Hence the lemma is proven.
\end{proof}

\begin{proof}[Proof of Proposition \ref{prop:nu_uguale}.]
The proof is based on the same induction argument of \cite{barcafus}*{Proposition 3.6}. 
We already observed in Remark \ref{rk:us} that condition \eqref{eq:condizioni_u} implies \eqref{eq:nu_ug_2}. 
Let us now prove the converse implication.
The case $k=1$ is proven in
\cite{CF}*{Proposition~ 2.3}.

\step{1} Let $k>1$ and let $v\in\sobo(\Omega)$ be a non-negative function satisfying
\eqref{eq:condizioni_u} and such that for $\Ha^{n-k+1}\aev (x,t)\in\esspjt(\mathcal{S}_v)$ the set $\{y:v(x,y)>t\}$ is
equivalent to a $k$-dimensional ball. For $i=1,\dotsc,k$, set
\[
 C^i:=\{(x,y)\in\Omega :\partial_{y_i} v(x,y)=0\}\cap\{(x,y)\in\Omega : \text{either\ }M(x)=0\text{\ or\ }
 v(x,y)<M(x)\}\,.
\]
We claim that for $v$ as above $\Hn(C^i)=0$.
Indeed, by Theorem \ref{thm:GMS2}, we see that the set
\[
 A^i=\{(x,y,t)\in\brid\mS_v: \nu_{y_i}^{\mS_v}=0\}\cap\{(x,y,t)\in\brid\mS_v : \text{either\ } M(x)=0\text{\ or\ }
t<M(x)\}
\]
satisfies 
\begin{equation}\label{eq:boh1}
\Hn(A^i)\geq\Hn(C^i)\,.\end{equation}
From Theorem \ref{thm:volpert}, up to $\Ha^{k-1}$ negligible sets, we get
\[
 A^i_{x,t}=\{y\in(\brid\mS_v)_{x,t}:\nu_{y_i}^{(\mS_v)_{x,t}}=0\}\cap\{(x,y,t)\in\brid\mS_v : \text{either\ }
M(x)=0\text{\ or\ } t<M(x)\}\,.
\]
Since almost every section of the subgraph of $v$ is a ball, we see that $\Ha^{k-1}(A^i_{x,t})=0$. Hence, using
\eqref{eq:boh1}, assumption \eqref{eq:condizioni_u} with Theorem \ref{thm:GMS2} and the coarea formula, we have
\[\begin{split}
\Hn(C^i)\leq\Hn(A^i)=\Hn(A^i\cap\{\nu_y^{\mS_v}\neq 0\})
=\int_{\projt(\brid\mS_v)}dx\,dt\int_{(\brid\mS_v)_{x,t}\cap A^i_{x,t}}\frac{d\Ha^{k-1}}{\abs{\nu_y^{\mS_v}}}=0\,,
\end{split}\]
and so the claim is proven.

\step{2} For $i=0,\dotsc,k$ define recursively $\Omega^0:=\Omega$, $\Omega^i:=(\Omega^{i-1})^{S_i}$, where $S_i$ is the
$1$-codimensional Steiner symmetrization with respect to $y_i$. The functions $u^i$ are defined accordingly. Assumption
\eqref{eq:equality} and Theorem \ref{thm:disuguaglianza} imply that
\[
 \int_{\Omegas}f(\nabla \us)\,dz=\int_{\Omega^{k-1}} f(\nabla u^{k-1})\,dz=\dotsb=\int_{\Omega^1} f(\nabla u^1)\,dz
 =\int_{\Omega}f(\nabla u)\,dz\,.
\]
Hence, by Lemma \ref{lemma:sezioni_palle}, we see that ${\mS}_{u^k}$ is equivalent to $\Sus$. From
\eqref{eq:nu_ug_2}
and \eqref{eq:boh1} we see that
\[
 \Hn\bigl(\{(x,y)\in\Omega^k : \nabla_{y_k} u^k (x,y)=0\}\cap\{(x,y)\in\Omega^k:
\text{either\ }M(x)=0\text{\ or\ }0<u^k<M(x)\}\bigr)=0\,.
\]
Since the assertion holds for $k=1$, we deduce
\[
 \Hn\bigl(\{(x,y)\in\Omega^{k-1}:\nabla_{y_{k-1}} u^{k-1}=0)\}\cap\{(x,y)\in\Omega^{k-1}:M(x)=0
\text{\
or\ }0<u^{k-1}<M(x)\}\bigr)=0
\]
and this clearly implies that 
\[
 \Hn\bigl(\{(x,y)\in\Omega^{k-1}:\nabla_{y} u^{k-1}=0\}\cap\{\text{either\ }M(x)=0\text{\ or\ }
0<u^{k-1}<M(x)\}\bigr)=0\,.
\]
The assertion now follows iterating this argument.
\end{proof}

\begin{proof}[Proof of Theorem \ref{thm:equality}.]
The first statement is Lemma \ref{lemma:sezioni_palle}, see also Remark \ref{rk:fsobo}.

By \eqref{eq:su_equiv_sus} it is sufficient to show that $(\Su)^\sigma$ is equivalent to $\Su$. From the previous
statement,
we know that for $\Ha^{n-k+1}\aev (x,t)\in\esspjt(\Su)$ every section of $(\Su)_{x,t}$ is
equivalent to a ball in $\Rk$ with radius $R(x,t)$ and denote by $b:\Rnk\times\Rt\to\R^{n+1}$ the center of this ball.
On replacing $u$ by $\us$ in Lemma \ref{lemma:sezioni_palle}, we see that for $\Ha^{n-k+1}\aev
(x,t)\in\esspjt((\Su)^\sigma)$ every $(x,t)$
section of $(\Su)^\sigma$ is equivalent to a ball of the same radius $R(x,t)$ and denote by
$\tilde{b}:\Rnk\times\Rt\to\R^{n+1}$
the center of the ball. From the very definition of the Steiner rearrangement we have that
$\tilde{b}(x,t)\equiv(x,0,t)$. Now it is sufficient to show that $b-\tilde{b}\equiv(0,c,0)$ for some $c\in\Rk$.

The case $k=1$ is \cite{CF}*{Theorem 2.2}. Let $k>1$ and for $i=1,\dotsc,k$ let $S_i$ be the Steiner symmetrization in
codimension $1$
with
respect to $y_i$. Clearly,
$\Omegas=(\Omegas)^{S_i}=(\Omega^{S_i})^\sigma$
and therefore \eqref{eq:disuguaglianza} implies
\begin{equation}\label{eq:pf_ug_1}
 \int_{\Omegas}f(\nabla \us)\,dz\leq\int_{\Omega^{S_i}}f(\nabla u^{S_i})\,dz\leq\int_\Omega f(\nabla u)\,dz\,,
\end{equation}
for $i=1,\dotsc,k$. From \eqref{eq:equality} we get equalities in \eqref{eq:pf_ug_1}. Since almost every
section $(\Su)_{x,t}$ is a ball, arguing as in Step 1 of the proof of Proposition \ref{prop:nu_uguale} we get  
\[
\Ha^n\bigl(\{z\in\Omega:\partial_{y_i}u(z)=0\} \cap
\{z\in\Omega:\text{ either }M(z')=0 \text{ or } u(z)<M(z')\}\bigr)=0\,,
\]
where $z':=(x,y_1,\dotsc,y_{i-1},y_{i+1},\dotsc y_k)$. Similarly, we also
get that
\[
 \Ha^{n-1}\bigl(\{z\in\brid\Omega : \nu^\Omega_{y_i}=0\}\cap\{\pi_{n-1}(\Omega)\times\R_{y_i}\}\bigr)=0\,,
\]
where $\pi_{n-1}$ is the projection on $z'$. Therefore, by the $k=1$ case, we have that $(b(x,t))_{y_1}\equiv c_1$ for
some
$c_1\in\R$. Now iterate the procedure and obtain $(b(x,t))_y\equiv(c_1,\dotsc,c_k)$ and so $b-\tilde{b}\equiv(0,c,0)$
with $c=(c_1,\dotsc,c_k)$.
\end{proof}

\section{\texorpdfstring{The $BV$ case}{The BV case}}\label{sec:bv}
In this section we are going to prove the Pólya-Szegő inequality for the Steiner rearrangement of a function of
 bounded variation and the characterization of the equality cases. As already observed in the introduction, we will
first prove analogous results for geometrical functionals depending on the generalized inner normal. In this setting, we
will first show a Pólya-Szegő principle in Theorem \ref{thm:F-ineq} an the characterization of the equality cases in
Theorem \ref{thm:equal_F}. 

Next two Lemmata will be used in the proof of Theorem \ref{thm:F-ineq}.
\begin{lemma}\label{lemma:F-ineq-1}
 Let $U\subset\Rnk\times\Rt$ be an open set. Let $F:\R^{n+1}\to [0,+\infty]$ be a convex function satisfying 
\eqref{eq:F-hom} and
\eqref{eq:F-radial} and let $E$ be a set of finite perimeter in $U\times\Ry$ such that
$\Ha^{n+1}(E\cap(U\times\Ry))<+\infty$. Then
\begin{equation}\label{eq:F-ineq-1}\begin{split}
 \int_{\brid\Es\cap(B\times\Ry)} F(\nu^{\Es})\, d\Hn & \leq
 \int_B \Ft \left( \frac{D_1 L}{\abs{DL}},\dotsc,\frac{D_{n-k} L}{\abs{DL}},0,\frac{D_t L}{\abs{DL}}\right)\,
d\abs{DL}\\
 &\quad+\Ft(0,\dotsc,0,1,0) \abs{D_y\chi_{\Es}}(B\times\Ry)
 \end{split}
\end{equation}
for every Borel set $B\subset U$.
\end{lemma}
\begin{proof}

% Moreover, since the assumptions made on $F$
% allow us to write
% \[
%  F(\xi)=\sup_{j\in\N} (\xi \cdot \alpha_j)^+\,,
% \]
% for every $\xi\in\R^{n+1}$ and for some sequence $\{\alpha_j\}\subset\R^{n+1}$, we may also assume that $F$ is bounded
% on $B$.

Without loss of generality we can assume that $B$ is a bounded open set.

\step{1} Let us prove inequality \eqref{eq:F-ineq-1} assuming that $F$ is everywhere finite, hence continuous.
By approximation we can find a sequence of
functions
%see e.g.\ \cite{EG}*{\S 5.2.2 Theorems 2 and 3} or \cite{BFP}*{\S 3.9 and 3.13}
$\{L_j\}\subset C^\infty(B)$ such that $L_j(x,t)>0$ for every $(x,t)\in B$, $L_j\to L$ in $L^1(B)$, $\nabla
L_j\,\Hn\weakto DL$ weakly* in the sense of measures and
\begin{equation}\label{eq:pr_lem_F1-1}
 \int_B \abs{\nabla L_j}\,dx\,dt\to \abs{DL}(B)\,.
\end{equation}
For $j\in\N$ define the sets $E_j:=\{(x,y,t): (x,t)\in B,\,\omega_k\abs{y}^k\leq L_j(x,t)\}.$ Then
$\chi_{E_j}\to\chi_{\Es}$ in $L^1(B\times\Ry)$ and since 
\[
 \abs{D\chi_{E_j}}(B\times\Ry)=P(E_j;B\times\Ry)\leq C\,,
\]
for some constant depending only on $B$, we deduce that
\begin{equation}\label{eq:pr_lem_F1-2}
 D\chi_{E_j}\weakto D\chi_{\Es}\text{\ weakly*\ in\ }B\times\Ry\,.
\end{equation}
Using the convexity of $F$, assumption \eqref{eq:F-hom} and \eqref{eq:DChi} we have
\begin{equation}\label{eq:pr_lem_F1-3}
\begin{split} 
&\int_{\brid{\Es}\protect\cap (B\times\Ry)} F(\nu^{\Es})\,d\Hn \\
&\quad\leq\int_{\brid{\Es}\cap (B\times\Ry)} \Ft(\nu_x^{\Es},0,\nu_t^{\Es})\,d\Hn+
\int_{\brid{\Es}\cap (B\times\Ry)} \Ft(0,\nu_y^{\Es},0)\,d\Hn\\
&\quad=\int_{B\times\Ry}
\Ft\left(\frac{D_x \chi_{\Es}}{\abs{D \chi_{\Es}}},0,\frac{D_t \chi_{\Es}}{\abs{D\chi_{\Es}}}\right)
d\abs{D\chi_{\Es}} 
+ \Ft(0,1,0) \int_{\brid{\Es}\cap (B\times\Ry)} \abs{\nu_y^{\Es}}\,d\Hn\,.
\end{split}
\end{equation}
Using \eqref{eq:pr_lem_F1-2}, Reshetnyak's lower semicontinuity Theorem (see, e.g., \cite{AFP}*{Theorem 2.38})
and \eqref{eq:DChi} we get
\begin{equation}\label{eq:pr_lem_F1-4}
 \begin{split}
\int_{B\times\Ry} \Ft\left(\frac{D_x \chi_{\Es}}{\abs{D \chi_{\Es}}},0,\frac{D_t \chi_{\Es}}{\abs{D\chi_{\Es}}}\right)
d\abs{D\chi_{\Es}}
&\leq \liminf_{j\to\infty} \int_{B\times\Ry}
\Ft\left(\frac{D_x \chi_{E_j}}{\abs{D \chi_{E_j}}},0,\frac{D_t \chi_{E_j}}{\abs{D\chi_{E_j}}}\right)
d\abs{D\chi_{E_j}}\\
&= \liminf_{j\to\infty} \int_{\brid E_j\cap(B\times\Ry)}\Ft(\nu_x^{E_j},0,\nu_t^{E_j})\,d\Hn\,.
 \end{split}
\end{equation}
Since the functions $L_j$ are smooth, for $i=1,\dotsc,n-k,t$
\[
 \nu_i^{E_j}(x,y,t)=\frac{\partial_i L_j(x,t)}{\sqrt{p_{j}(x,t)^2+\abs{\nabla L_j(x,t)}^2}}
\]
for every $(x,y,t)\in\brid E_j\cap(B\times\Ry)$, where $p_{j}(x,t)$ stands for the perimeter of ${(E_j)}_{x,t}$. Using
this
equality with \eqref{eq:pr_lem_F1-3},
\eqref{eq:pr_lem_F1-4} and \eqref{eq:DChi} we see that
\begin{equation}\label{eq:pr_lem_F1-5}
 \begin{split}
&\int_{\brid\Es\cap(B\times\Ry)} F(\nu^{\Es})\,d\Hn\\
&\quad \leq\liminf_{j\to\infty} \int_{\brid E_j\cap(B\times\Ry)} 
F\Biggl(\frac{\partial_i L_j}{\sqrt{p_j^2+\abs{\nabla L_j}^2}},0,\frac{\partial_t L_j}{\sqrt{p_j^2+\abs{\nabla
L_j}^2}}\Biggr)\,d\Hn \\
&\qquad  + \Ft(0,1,0) \int_{\brid\Es\cap(B\times\Ry)}\abs{\nu_y^{\Es}}\,d\Hn\\
&\quad =\liminf_{j\to\infty}\int_B \Ft(\nabla_x L_j,0,\partial_t L_j)\,dx\,dt
+\Ft(0,1,0)\abs{D_y\chi_{\Es}}(B\times\Ry)\\
&\quad =\liminf_{j\to\infty}\int_B \Ft\left(\frac{\nabla_x L_j}{\abs{\nabla L_j}},0,\frac{\partial_t L_j}{\abs{\nabla
L_j}}\right) \abs{\nabla L_j}\,dx\,dt
+\Ft(0,1,0)\abs{D_y\chi_{\Es}}(B\times\Ry)\,.
 \end{split}
\end{equation}
Since $\nabla L_j\,\Hn\weakto DL$ weakly* and \eqref{eq:pr_lem_F1-1} holds, we can apply Reshetnyak's continuity Theorem
(see, e.g., \cite{AFP}*{Theorem 2.39}) and get
\begin{equation}\label{eq:pr_lem_F1-6}
\liminf_{j\to\infty}\int_B \Ft\left(\frac{\nabla_x L_j}{\abs{\nabla L_j}},0,\frac{\partial_t L_j}{\abs{\nabla
L_j}}\right)
\abs{\nabla L_j}\,dx\,dt
= \int_B \Ft\left(\frac{D_x L}{\abs{D L}},0,\frac{D_t L}{\abs{D L}}\right)
\,d\abs{DL}\,.
\end{equation}
Then, inequality \eqref{eq:F-ineq-1} follows combining \eqref{eq:pr_lem_F1-5} and \eqref{eq:pr_lem_F1-6}.

\step{2} Let us remove the additional assumption made in Step 1. Since $F$ is a convex function
satisfying \eqref{eq:F-hom} and \eqref{eq:F-radial}, we see that there exists a sequence
$\{(a_j,b_j,c_j)\}\subset\Rnk\times\R\times\R$ such that
\[
 F(\xi)=  \sup_{j\in\N}\,\bigl\{(\xi_x\cdot a_j+\abs{\xi_y} b_j+\xi_t c_j)^+\bigr\}\,,
\]
for every $\xi\in\R^{n+1}$. Define, for $N\in \N$,
\[
 F_N(\xi):=\sup_{1\leq j\leq N} \bigl\{(\xi_x\cdot a_j+\abs{\xi_y} b_j+\xi_t c_j)^+\bigr\}\,.
\]
Note that $F_N$ is a continuous function and satisfies \eqref{eq:F-hom} and \eqref{eq:F-radial}. Since
$F_N(\xi)\nearrow F(\xi)$ pointwise monotonically, inequality 
\eqref{eq:F-ineq-1} follows applying Step 1 to the functions $F_N$ and using the monotone convergence theorem.
\end{proof}

The following lemma gives informations on the gradient of the function $L$. It is a simple
variant of \cite{CCF}*{Lemmata 3.1 and 3.2}.
\begin{lemma}\label{lemma:DL}
Let $U\subset\Rnk\times\Rt$ be an open set and let $E$ be a set of finite perimeter in $U\times\Ry$ such that
$\Ha^{n+1}(E\cap(U\times\Ry))<+\infty$. Then $L\in BV(U)$ and for any bounded Borel function g in $U$ 
\begin{equation}\label{eq:lemma_DL}
\int_{U} g(x)\,dD_i L(x)=\int_{U\times\Ry}g(x)\,dD_i\, \chi_E(x,y)\,,\quad \text{\ for\ }i=1,\dotsc,n-k,t\,.
\end{equation}
\end{lemma}

\begin{lemma}\label{lemma:F-ineq-2}
 Let $U\subset\Rnk\times\Rt$ be an open set and let $F:\R^{n+1}\to [0,+\infty]$ be a convex function satisfying
\eqref{eq:F-hom}. Let $E$ be a set of finite perimeter in $U\times\Ry$ such that $\Ha^{n+1}(E\cap(U\times\Ry))<+\infty$.
Then
\begin{equation}\label{eq:F-ineq-2}
 \int_B \Ft \left( \frac{D_1 L}{\abs{DL}},\dotsc,\frac{D_{n-k} L}{\abs{DL}},0,\frac{D_t L}{\abs{DL}}\right)\,
d\abs{DL} \leq
\int_{\brid E \cap(B\times\Ry)}\Ft(\nu^E_1,\dotsc,\nu^E_{n-k},0,\nu^E_t)\, d\Hn
\end{equation}
for every Borel set $B\subset U$.
\end{lemma}
\begin{proof}
 As in the previous proof, we can assume that $B$ is a bounded open set. Since $F$ is a non-negative convex function
satisfying \eqref{eq:F-hom}, there exists a sequence of vectors $\{\alpha_j\}\in\Rnk\times\Rt$
such that
\begin{equation}\label{eq:bah}
 F(\xi_x,0,\xi_t)=\sup_{j\in\N}\left\{(\alpha_j\cdot \xi_{x,t})^+\right\}
\end{equation}
for every $\xi\in\R^{n+1}$, where $\xi_{x,t}=(\xi_x,\xi_t)\in\R^{n-k+1}$. Hence we deduce that (see, e.g.,
\cite{AFP}*{Lemma 2.35})
\begin{equation}\label{eq:pr_F_ineq_2-2}
 \int_B \Ft\left(\frac{D_x L}{\abs{DL}},0,\frac{D_t L}{\abs{DL}}\right)\,d\abs{D L}
 =\sup\biggl\{\sum_{j\in J}\int_{B_j} \Bigl(\alpha_j\cdot\frac{DL}{\abs{DL}}\Bigr)^+\, d\abs{DL}\biggr\}\,,
\end{equation}
where the supremum is extended over all finite sets $J\subset\N$ and all families $\{B_j\}_{j\in J}$ of pairwise
disjoint Borel subsets of $B$. For a fixed family $\{B_j\}_{j\in J}$ and a fixed $j\in\N$ let us define
\begin{equation*}
 P_j:=\biggl\{(x,t)\in B_j : \alpha_j\cdot \frac{DL}{\abs{DL}}(x,t)\geq 0\biggr\}\,.
\end{equation*}
Hence, on applying \eqref{eq:lemma_DL}, we get
  \[\begin{split}
 \int_{B_j} \Bigl(\alpha_j\cdot \frac{DL}{\abs{DL}}\Bigr)^+d\abs{DL}
 &=\int_U \chi_{P_j}(x,t)\biggl(\sum_{i=1}^{n-k}(\alpha_j)_i\frac{D_i L}{\abs{DL}}+(\alpha_j)_t \frac{D_t
L}{\abs{DL}}\biggr)\,d\abs{DL}\\
&\nqquad\nquad = \sum_{i=1}^{n-k}\int_U (\alpha_j)_i\, \chi_{P_j}(x,t)\,dD_i L(x,t)
+\int_U \!\!(\alpha_j)_t\, \chi_{P_j}(x,t)\,dD_t L(x,t)\\
&\nqquad\nquad =\sum_{i=1}^{n-k} \int_{U\times\Ry}(\alpha_j)_i\,\chi_{(P_j\times\Ry)}(x,y,t)\,dD_i \chi_E
+ \int_{U\times\Ry}\!\!(\alpha_j)_t\,\chi_{(P_j\times\Ry)}(x,y,t)\,dD_t \chi_E\,.
\end{split}\]
Combining the last equality with \eqref{eq:DChi} we have
\[
 \int_{B_j}\Bigl(\alpha_j\cdot\frac{DL}{\abs{DL}}\Bigr)^+\,d\abs{DL} =
 \int_{\brid E} \chi_{(P_j\times\Ry)} \alpha_j\cdot\nu_{x,t}^E\,d\Hn
 \leq \int_{\brid E}
 \chi_{(B_j\times\Ry)}\bigl(\alpha_j\cdot\nu_{x,t}^E\bigr)^+\,d\Hn\,.
\]
Hence, on using \eqref{eq:bah} we see that
\[
 \begin{split}
 \sum_{j\in J}\int_{B_j} \bigl(\alpha_j\cdot \frac{DL}{\abs{DL}}\bigr)^+d\abs{DL}
 \leq \sum_{j\in J}\int_{\brid E \cap(B_j\times\Ry)} \!\!\!\Ft(\nu_x^E,0,\nu^E_t)\,d\Hn
 \leq\int_{\brid E\cap(B\times\Ry)}\!\!\! \Ft(\nu_x^E,0,\nu_t^E)\,d\Hn\,.
\end{split}\]
Then, combining \eqref{eq:pr_F_ineq_2-2} and the last inequality, we get \eqref{eq:F-ineq-2}.
\end{proof}

\begin{lemma}\label{lemma:L_loc_fin}
Let $U\subset\Rnk\times\Rt$ be an open set and let $E$ be a set of finite perimeter in $U\times\Ry$ such that 
$L(x,t)<+\infty$ for $\Hn\aev(x,t)\in U$. Then, for every open set $U'\Subset U$ 
\begin{equation}\label{eq:L_loc_fin}
 \Ha^{n+1}(E\cap(U'\times\Ry))<+\infty\,.
\end{equation}
\end{lemma}
\begin{proof}
 Given an open set $U'\Subset U$ define
 \[E_h=E\cap \bigl(U'\times B(0,h)\bigr)\,\text{\ for\ }h\in\N\,.\]
 Without loss of generality, let us assume that $\partial U'$ is smooth. Since $E_h$ has finite perimeter in
$U'\times\Ry$, then by \eqref{eq:measure_bdry} we see that
\begin{equation}\label{eq:pr_L_locf1}
 \measbdry E_h\cap(U'\times\Ry)\subset\left(\measbdry E\cup\{\abs{y}=h\}\right)\cap(U'\times\Ry)\,.
\end{equation}
Since $\Ha^{n+1}(E_h\cap(U'\times\Ry))<+\infty$, arguing as in the proof of Lemma \ref{lemma:F-ineq-1}
and using \eqref{eq:pr_L_locf1}, \eqref{eq:DChi} and \eqref{AZZ_(3.5)} we deduce that
\begin{equation*}
 P\bigl((E_h)^\sigma;U'\times\Ry\bigr)\leq P\bigl(E_h;U'\times\Ry\bigr)\leq C\,,
\end{equation*}
for some constant $C$ depending only on $U'$.
Define  $m_h=\dashint_{U'} L_h(x,t)\,dx\,dt$, where $L_h(x,t)$ stands for $\Ha^{n-k+1}((E_h)_{x,t})$. Using the Poincaré
inequality for functions of bounded variations (see, e.g., \cite{AFP}*{Theorem 3.44}) we have that
\begin{equation}\label{eq:pr_L_locf3}
 \int_{U'}\abs{L_h(x,t)-m_h}\,dx\,dt\leq C\,\abs{DL_h}(U')\leq C\, P\bigl((E_h^\sigma);U'\times\Ry\bigr)\\
 \leq C\,,
\end{equation}
for some constant $C$ depending only on $U'$. Up to subsequences, we have that $m_h\to m$ for some $m\in [0,+\infty]$. 
As $L_h(x,t)\to L(x,t)$ for $\Hn\aev(x,t)\in U'$, using \eqref{eq:pr_L_locf3} and Fatou's Lemma we infer that
\[
 \int_{U'}\abs{L(x,t)-m}\,dx\,dt\leq C\,.
\]
Since $L(x,t)$ is finite for $\Hn\aev(x,t)\in U'$, the last equation gives $m<+\infty$ and $L(x,t)\in L^1(U')$.
Hence, \eqref{eq:L_loc_fin} follows.
\end{proof}

\begin{theorem}\label{thm:F-ineq}
 Let $F:\R^{n+1}\to [0,+\infty]$ be a convex function satisfying \eqref{eq:F-hom} and \eqref{eq:F-radial}. Let
$U\subset \R^{n-k}\times\Rt$ be an open set and let $E$ be a set of finite perimeter in $U\times\Ry$ such that
$L(x,t)<+\infty$ $\Ha^{n-k+1}\aev\text{in\ }U$. Then
\begin{equation}\label{eq:F-ineq}
 \int_{\brid\Es\cap(B\times\Ry)} F(\nu^{\Es})\,d\Hn\leq \int_{\brid E\cap(B\times\Ry)} F(\nu^E)\,d\Hn
\end{equation}
for every Borel set $B\subset U$. In particular, if $E$ is a set of finite perimeter in $\R^{n+1}$, then
 \begin{equation}\label{eq:F-ineq-bis}
\int_{\brid\Es} F(\nu^{\Es})\,d\Hn\leq \int_{\brid E} F(\nu^E)\,d\Hn\,.
 \end{equation}
\end{theorem}
\begin{proof}
% On replacing, if necessary, $F(\xi)$ by $F(\xi)+\delta\abs{\xi}$, we may assume that 
% \begin{equation}\label{eq:F_grea_nu}
% F(\xi)\geq \gamma\abs{\xi}\,,\quad\text{for
% some constant\ }\gamma>0\,.
% \end{equation}
% The general case is then deduced letting $\delta$ going to $0$.

\step{1}
Let us first assume that $\Ha^{n+1}(E\cap(U\times\Ry))<+\infty$.
Let $G_{\Es}$ be the set given by Vol'pert's Theorem \ref{thm:volpert}. For any
Borel set $B\subset U$ define $B_1=B\setminus G_{\Es}$ and $B_2=B\cap G_{\Es}$.

By inequalities \eqref{eq:F-ineq-1} and \eqref{eq:F-ineq-2} we see that
\begin{equation}\label{eq:pr_F_ineq-1}
 \int_{\brid \Es\cap(B_1\times\Ry)} F(\nu^{\Es})\,d\Hn
 \leq \int_{\brid \Es\cap(B_1\times\Ry)} F(\nu^E)\,d\Hn
 + \Ft(0,1,0)\abs{D_y\chi_{\Es}}(B_1\times\Ry)\,.
\end{equation}
Moreover, by \eqref{eq:DChi}, coarea formula \eqref{eq:coarea} and (ii) of Theorem \ref{thm:volpert} we get
\begin{equation}\label{eq:pr_F_ineq-2}
 \abs{D_y \chi_{\Es}}(B_1\times\Ry)
 =\int_{\brid \Es\cap(B_1\times\Ry)} \abs{\nu_y^{\Es}}\,d\Hn
  =\int_{B_1}\Ha^{k-1}\bigl(\brid \Es_{x,t}\bigr)\,dx\,dt
 =0\,,
\end{equation}
where the last equality holds since $\Hn(\esspjt(E)\cap B_1)=0$. Hence, \eqref{eq:pr_F_ineq-1} and
\eqref{eq:pr_F_ineq-2} give
\begin{equation}\label{eq:pr_F_ineqB1}
 \int_{\brid\Es\cap(B_1\times\Ry)} F(\nu^{\Es})\,d\Hn\leq \int_{\brid E\cap(B_1\times\Ry)} F(\nu^E)\,d\Hn\,.
\end{equation}

For all $(x,t)\in B_2$, we have $\nu_y^{\Es}\neq 0\,\,\,\Ha^{k-1}\aev$on $\partial \Es_{x,t}$. Hence, since $\Es_{x,t}$
is a ball, we get that indeed $\nu^{\Es}_y\neq 0$ at all point on $\partial \Es_{x,t}$. Therefore, $\nu^{\Es}_y\neq 0$
for all point on $\brid \Es\cap (B_2\times\Ry)$ and we can apply the coarea formula, thus getting
\begin{equation}\label{eq:catena1}\begin{aligned}
\int_{\brid\Es\cap(B_2\times\Ry)} & F(\nu^{\Es})\,d\Hn && \\
&\nqquad\nqquad\nqquad=\int_{\brid\Es\cap(B_2\times\Ry)}
\Ft\left(\frac{\nu^{\Es}}{\abs{\nu_y^{\Es}}}\right)\abs{\nu_y^{\Es}}\,d\Hn
&&\text{by \eqref{eq:F-hom} and \eqref{eq:F-radial}}\\
&\nqquad\nqquad\nqquad=\int_{B_2}dx\,dt\int_{\brid(\Es)_{x,t}}\Ft\left(\frac{\nu_x^{\Es}}{\abs{\nu_y^{\Es}}},1,
\frac{\nu_t^{\Es}}{\abs{\nu_y^{\Es}}}\right)d\Ha^{k-1}(y)&& \text{by \eqref{eq:coarea}}\\
&\nqquad\nqquad\nqquad=\int_{B_2}\Ft\left(\nabla_x L(x,t),\Ha^{k-1}(\brid \Es_{x,t}),\partial_t L(x,t)\right)dx\,dt
&&\text{by \eqref{eq:lemma_partialLEs}.}\\
&\nqquad\nqquad\nqquad\leq 
\int_{B_2}\Ft\left(\nabla_x L(x,t),\Ha^{k-1}(\brid E_{x,t}),\partial_t L(x,t)\right)dx\,dt
&&\text{by the isoperimetric inequality}\,.
\end{aligned}
\end{equation}

Since $F$ is a non-negative convex function satisfying \eqref{eq:F-hom} and \eqref{eq:F-radial}, we see that there
exists a sequence of vectors $\{(\xi_h,\rho_h,\tau_h)\}\subset\Rnk\times\R\times\R$ such that
\[
 \Ft(x,r,t)=\sup_{h\in\N}\bigl\{(x\cdot \xi_h+r\rho_h+t\tau_h)^+\bigr\}\,.
\]
Hence, we deduce that (see, e.g., \cite{AFP}*{Lemma 2.35})
\begin{equation*}
 \int_{B_2}\!\!\Ft\left(\nabla_x L(x,t),\Ha^{k-1}(\brid E_{x,t}),\partial_t L(x,t)\right)dx\,dt=
 \sup\Biggl\{\sum_{h\in H}\int_{A_h}\!\bigl(\nabla_x L\cdot \xi_h + p(x,t)\, \rho_h +\partial_t L\,\tau_h
 \bigr)^+\Biggr\}\,,
\end{equation*}
where $p(x,t):=\Ha^{k-1}(\brid E_{x,t})$ and the supremum is extended over all finite sets $H\subset\N$ and all families
$\{A_h\}_{h\in H}$ of pairwise disjoint Borel subsets of $B_2$.  For a fixed family $\{A_h\}_{h\in H}$ and a fixed
$h\in\N$, define
\[
 P_h:=\bigl\{(x,t)\in A_h : \nabla_x L(x,t)\cdot \xi_h + p(x,t)\, \rho_h +\partial_t L(x,t)\,\tau_h \geq 0\bigr\}\,.
\]
Let us define 
\[g(x,t):=\int_{\brid E_{x,t}}\frac{\nu^E_{x,t}(x,y,t)}{\abs{\nu_y^E(x,y,t)}} d\Ha^{k-1}(y)\,.
\]
From \eqref{eq:formula_deriv} and 
considering that $DL$ is absolutely continuous on $B_2$, setting $\Ath:=A_h\cap P_h$, we have
\begin{equation}\label{eq:catena2}\begin{split}
 &\sum_{h\in H}\int_{A_h}\bigl(\nabla_x L(x,t)\cdot \xi_h+ p(x,t)\rho_h + 
 \partial_t L(x,t)\tau_h\bigr)^+ \\
 & =\sum_{h\in H}\int_{\Ath} \nabla_x L(x,t)\cdot \xi_h+ p(x,t)\rho_h + 
 \partial_t L(x,t)\tau_h \\
 &=\sum_{h\in H} \Biggl[
 \int_{\brid E\cap(\Ath\times\Rk)\cap\{\nu_y^E=0\}}(\xi_h,\tau_h)\cdot \nu^E_{x,t}(x,y,t)\,d\Hn+
 \int_{\Ath} g(x,t)\cdot (\xi_h,\tau_h)+p(x,t)\rho_h \,dx\,dt\Biggr]\\
 &\leq \sum_{h\in H}\Biggl[
 \int_{\brid E\cap(\Ath\times\Rk)\cap\{\nu_y^E=0\}} \Ft(\nu^E_x,0,\nu^E_t)\,d\Hn\\
 &\phantom{\leq \sum_{h\in H}\}}\,+
 \int_{\Ath}
 \Ft\left(
\int_{\brid E_{x,t}} \frac{\nu_x^E}{\abs{\nu_y^E}}\,d\Ha^{k-1},
\int_{\brid E_{x,t}}d\Ha^{k-1},
\int_{\brid E_{x,t}} \frac{\nu_t^E}{\abs{\nu_y^E}}\,d\Ha^{k-1}\right)dx\,dt\Biggr]\\
&\leq \sum_{h\in H}\Biggl[
\int_{\brid E\cap(A_h\times\Rk)\cap\{\nu_y^E=0\}} F(\nu^E)\,d\Hn
+ \int_{A_h}dx\,dt\int_{\brid E_{x,t}}
\Ft\left(\frac{\nu_x^E}{\abs{\nu_y^E}},1,\frac{\nu_t^E}{\abs{\nu_y^E}}\right) d\Ha^{k-1}(y)\Biggr]
=:\mathcal{J}
\,,
\end{split}\end{equation}
where the last inequality is due to Jensen's inequality. On applying the coarea formula, we
see that 
\begin{equation}\label{eq:catena3}\begin{split}
\mathcal{J}&=\sum_{h\in H}\Biggl[
\int_{\brid E\cap(\Ath\times\Rk)\cap\{\nu_y^E=0\}} F(\nu^E)\,d\Hn
+ \int_{\brid E\cap (\Ath\times\Rk)\cap\{\nu^E_y\neq 0\}}
F(\nu)\,d\Hn\Biggr]\\
&\leq \sum_{h\in H}\Biggl[
\int_{\brid E\cap(A_h\times\Rk)\cap\{\nu_y^E=0\}} F(\nu^E)\,d\Hn
+ \int_{\brid E\cap (A_h\times\Rk)\cap\{\nu^E_y\neq 0\}}
F(\nu)\,d\Hn\Biggr]\\
&=\int_{\brid E\cap(B_2\times\Rk)}F(\nu)\,d\Hn\,.
\end{split}\end{equation}

Now
inequality \eqref{eq:F-ineq} follows combining
\eqref{eq:pr_F_ineqB1}$-$\eqref{eq:catena3}.

\step{2} If the set $E$ is such that $L(x,t)<+\infty$ for $\Ha^{n-k+1}\aev (x,t)\in U$, then \eqref{eq:F-ineq} follows 
from Step 1 and from Lemma \ref{lemma:L_loc_fin}. 

\step{3} It remains to prove \eqref{eq:F-ineq-bis}. If $E$ has finite perimeter in $\R^{n+1}$, then the isoperimetric
inequality
(see, e.g., \cite{AFP}*{Theorem 3.46}) assures that either $E$ or $\R^{n+1}\setminus E$ has finite measure. In the first
case \eqref{eq:F-ineq-bis} is proven by the above calculations taking $U=\R^{n-k+1}$. In the second one,
\eqref{eq:F-ineq-bis} trivially holds, since $\Es$ is equivalent to $\R^{n+1}$ and so $\brid \Es=\emptyset$. 
\end{proof}

In order to prove Theorem \ref{thm:equal_bv} we need some results for the equality cases in \eqref{eq:F-ineq} and
\eqref{eq:F-ineq-bis}. For this, we
need to strengthen the assumptions. Namely, we require that for every $(x,t)\in\R^{n-k+1}$ and for every
$s_1,s_2\in\R^+$ with $s_1<s_2$,
\begin{equation}\label{eq:Fincr}
 \Ft(x,s_1,t)<\Ft(x,s_2,t)\,,
\end{equation}
whenever the right-hand side is finite.
\begin{proposition}\label{thm:equal_F}
 Let $F:\R^{n+1}\to[0,+\infty]$ be a convex function satisfying \eqref{eq:F-hom}, \eqref{eq:F-radial} and
\eqref{eq:Fincr} and let
$U\subset \Rnk\times\Rt$ be an open set. Let $E$ be a set of finite perimeter in $U\times\Ry$ such that $L(x,t)<+\infty$
$\Ha^{n-k+1}\aev\text{in}$ $U$. If
\begin{equation}\label{eq:equal_F}
 \int_{\brid E^\sigma\cap(U\times\Ry)} F(\nu^{E^\sigma})\,d\Hn=\int_{\brid E \cap(U\times\Ry)} F(\nu^E)\,d\Hn
 <\infty\,,
\end{equation}
then for $\Ha^{n-k+1}\aev (x,t)\in \esspjt(E)\cap U$ the section $E_{x,t}$ is equivalent to a $k$-dimensional ball.
\end{proposition}
\begin{proof}
 Assumption \eqref{eq:equal_F} and inequality \eqref{eq:F-ineq} assure us that
\begin{equation}\label{eq:pr_F_palla}
  \int_{\brid E^\sigma\cap(B\times\Ry)} F(\nu^{E^\sigma})\,d\Hn=
  \int_{\brid E\cap(B\times\Ry)} F(\nu^{E})\,d\Hn
\end{equation}
for every Borel set $B\subset U$. Possibly replacing $U$ by $U'$, where $U'\Subset U$, from Lemma
\ref{lemma:L_loc_fin}
we can assume that $\Ha^{n+1}(E\cap(U\times\Ry))<+\infty$. Hence, on choosing $B=U\cap G_E\cap G_{E^\sigma}$ in
\eqref{eq:pr_F_palla} we have equalities in \eqref{eq:catena1}. This, in combination with assumption \eqref{eq:Fincr}
and the fact that the integrals in \eqref{eq:equal_F} have finite value,
gives us that $\Ha^{k-1}(\brid E_{x,t})=\Ha^{k-1}(\brid E^\sigma_{x,t})$ for $\Ha^{n-k+1}\aev (x,t)\in B$ and
therefore for
$\Ha^{n-k+1}\aev (x,t)\in\esspjt(E)\cap U$. On applying the isoperimetric theorem the result is proven.
\end{proof}

Theorem \ref{thm:F-ineq} and Proposition \ref{thm:equal_F} are sufficient to prove Theorem \ref{thm:equal_bv}. The
problem of whether a set satisfying \eqref{eq:equal_F} is necessarily Steiner symmetric or not is the content of the
next result. Here, we need stronger assumptions. In particular we require that the precise representative $L^*$ of
$L$ ---  see, e.g., \cite{EG}*{\S 1.7.1}
for the
definition --- satisfies
\begin{equation}\label{eq:prec_repr}
 L^*(x,t)>0\text{ for }\Ha^{n-k-1}\aev (x,t)\in U\,.
\end{equation}
We introduce the following notation. Given $i=1,\dotsc,n-k$, for $(x,t)\in\Rnk\times\Rt$ we write
$\hat{x}_i:=(x_1,\dotsc,x_{i-1},x_{i+1},\dotsc,x_{n-k},t)$ and $\hat{t}:=x$. If $g$ is a function defined on an open set
$U\subset\Rnk\times\Rt$, we set $g_{\hat{x}_i}:=f_{|U\cap R_{\hat{x}_i}}$, where $R_{\hat{x}_i}$ is the straight line
passing through $(x_1,\dotsc,x_{i-1},0,x_{i+1},\dotsc,x_{n-k},t)$ and orthogonal to the hyperplane $x_i=0$. Then
$f_{\hat{t}}$ is defined accordingly.
\begin{theorem}\label{thm:SteinerF}
Let $F:\R^{n+1}\to[0,+\infty)$ be a convex function satisfying \eqref{eq:F-hom}, \eqref{eq:F-radial} and
\eqref{eq:Fincr}.
Let $U\subset\Rnk\times\Rt$ be an open set and let $E$ be a set of finite perimeter satisfying
\eqref{eq:prec_repr} and such that
\begin{equation}\label{eq:Lfinite}
L(x,t)<+\infty \text{ for } \Ha^{n-k+1}\aev(x,t)\in U\,.
\end{equation}
Assume that there exists a convex set $K\subset \Rnk\times\Rt$ such that the function
\begin{equation}\label{eq:Kconvex}\begin{split}
K\ni(\xi_x,\xi_t)\mapsto \Ft(\xi_x,1,\xi_t)\text{ is strictly convex} \text{ and } \\
\left(
\frac{\nu^E_x}{\abs{\nu^E_y}},\frac{\nu^E_t}{\abs{\nu^E_{y}}}
\right) \in K\,\,\, \Hn\aev \text{on }\brid E\cap (U\times\Rk)\,.
\end{split}\end{equation}
Assume also that
\begin{equation}\label{eq:abc}
 \Hn\left(\bigl\{(x,y,t)\in\brid \Es : \nu_y^{\Es}(x,y,t)=0\bigr\}\cap \bigl(U\times\Ry)\right)=0\,.
\end{equation}
If \eqref{eq:equal_F} is fulfilled, then for each connected component $U_\alpha$ of $U$, $E\cap(U_\alpha\times\Ry)$
is equivalent to $E^\sigma\cap
(U_\alpha\times\Ry)$ up to translations in the
$y$-plane. In particular, if $U$ is connected and $\Ha^{n-k+1}(\esspjt(E)\setminus U)=0$, then $E$ is equivalent
to $E^\sigma$ up to translations in the $y$-plane.
\end{theorem}
\begin{proof}
\step{1} Let $U_\alpha$ be any connected component of $U$. From Proposition \ref{thm:equal_F} we know that for
$\Ha^{n-k+1}\aev
(x,t)\in\esspjt(E)\cap U_\alpha$ the section $E_{x,t}$ is equivalent to a $k$-dimensional ball of radius $R(x,t)$ and
clearly the same holds for $E^\sigma$ with the same radius. Denote by $b(x,t)$ and $\tilde{b}(x,t)$ the center of
these balls. Since $E^\sigma$ is Steiner symmetric we have that $\tilde{b}(x,t)\equiv(x,0,t)$. The result will follow if
we show that $\beta(x,t):=\bigl(b(x,t)\bigr)_y$ is constant. Notice that $\beta(x,t)$ is a measurable function 
which, by \eqref{eq:prec_repr} and \eqref{eq:Lfinite} is finite a.e., and is equal to
\[
 \beta(x,t)= \frac{1}{L(x,t)}\int_{E_{x,t}}y\,dy\,.\\
\]

\step{2} Since equality \eqref{eq:equal_F} holds, arguing as in the proof
of Proposition \ref{prop:nu_uguale} we deduce
that condition \eqref{eq:abc} is equivalent to
\begin{equation}\label{eq:abc2}
 \Hn\left(\bigl\{(x,y,t)\in\brid E : \nu_y^{E}(x,y,t)=0\bigr\}\cap \bigl(U\times\Ry)\right)=0\,.
\end{equation}
Therefore, using \cite{barcafus}*{Theorem 4.3} we get that the function $\beta_{\hat{x}_i}\in W^{1,1}_\loc(U\cap
R_{\hat{x}_i};\Rk)$ and for $\Ha^1\aev x_i\in U\cap R_{\hat{x}_i}$
\begin{equation}\label{eq:bary}
 \beta'_{\hat{x}_i}(x_i)=\frac{1}{L^*_{\hat{x}_i}(x_i)}\int_{\brid E_{x,t}}[y-\beta_{\hat{x}_i}(x_i)]
 \frac{\nu^E_i(x,y,t)}{\abs{\nu^E_y(x,y,t)}}d\Ha^{k-1}(y)\,.
\end{equation}
A similar equality holds for $\beta'_{\hat{t}}(t)$.

By \eqref{eq:pr_F_palla} we have equalities in \eqref{eq:catena1} and \eqref{eq:catena2}. Hence, from \eqref{eq:abc2} we
get
\[
 \Ft\left(\int_{\brid E_{x,t}}\frac{\nu^E_x}{\abs{\nu^E_y}}d\Ha^{k-1},
 \int_{\brid E_{x,t}}d\Ha^{k-1},
 \int_{\brid E_{x,t}}\frac{\nu^E_t}{\abs{\nu^E_y}}d\Ha^{k-1}\right) =
 \int_{\brid E_{x,t}} F\left(\frac{\nu^E_x}{\abs{\nu^E_y}},1,\frac{\nu^E_t}{\abs{\nu^E_y}}\right)\,.
\]
From \eqref{eq:Kconvex}, $\nu^E_{x,t}/\abs{\nu^E_y}$ is constant with respect to $y$. Moreover, as $\brid E_{x,t}$
is a sphere, $\abs{\nu^E_y}$ is constant and so $\nu^E_{x,t}$ is constant. Hence, from \eqref{eq:bary} we get
\begin{equation}\label{eq:betaconst}
 \beta'_{\hat{x}_i}(x_i)=\frac{1}{L^*_{\hat{x}_i}(x_i)}\frac{\nu^E_i(x,t)}{\abs{\nu^E_y(x,t)}}\int_{\brid E_{x,t}}
 [y-\beta_{\hat{x}_i}(x_i)]\,d\Ha^{k-1}(y)=0\,,
\end{equation}
where we dropped the variable $y$ for functions that are constant in $\brid E_{x,t}$ and the last equality is due to the
definition of the function $\beta$. 

\step{3} We claim that $\beta$ is constant. Indeed, if $\beta$ is bounded, it is locally integrable. Therefore,
$\beta\in L^1_\loc(U_\alpha;\Rk)$ and its
restrictions $\beta_{\hat{x}_i}$ and $\beta_{\hat{t}}$ are absolutely continuous and integrable. Hence, by a standard
characterization of Sobolev functions (see, e.g.,
\cite{EG}*{\S 4.9, Theorem 2}) we have that $\beta\in W^{1,1}_{\loc}(U_\alpha;\Rk)$ and $\nabla \beta=0$ in $U_\alpha$
and so $\beta$ is constant in $U_\alpha$. For $\beta=(\beta_1,\dotsc,\beta_k)$ unbounded, fix $T>0$ and define the
truncated function $\beta^T$ as
\[\beta^T_j(x,t):=\begin{cases}
   \beta_j(x,t) &\text{ if }\abs{\beta_j(x,t)}\leq T\\
   T &\text{ if }\beta_j(x,t)>T\\
   -T &\text{ if }\beta_j(x,t)<-T\,,
  \end{cases}
\]
for $j=1,\dotsc,k$. Hence 
\[
  (\beta_{j,\hat{x}_i}^T)'=
  \begin{cases}
  0 &\text{ if }\abs{\beta_j(x,t)}> T\\
  \beta_{j,\hat{x}_i}' &\text{ if }\abs{\beta_j(x,t)}\leq T\,,
  \end{cases}
\]
with a similar equality holding for $(\beta_{j,\hat{t}_i}^T)'$.
Therefore, since $\beta^T$ is bounded, from \eqref{eq:betaconst} and the previous equality we deduce that $\beta^T=C^T$
a.e. for some constant $C^T\in\Rk$. Finally, as 
\[
 \beta(x,t)=\lim_{T\to +\infty}\beta^T(x,t)=\lim_{T\to\infty}C^T
\]
and since $\beta$ is finite a.e., we deduce that $\beta$ is constant.
\end{proof}

After proving the results concerning functionals of the form \eqref{eq:Fnu}, we deal now with the Pólya-Szeg\H{o}
principle for $BV$ functions. In the proof of Theorem \ref{thm:ineq-bv} we will use Theorem \ref{thm:frelax} below, a
consequence of relaxation results concerning $BV$ functions, see e.g., \cite{AFP}*{Theorem 5.47}.

\begin{theorem}[\cite{CF}*{Theorem F}]\label{thm:frelax}
 Let $f$ be a convex function satisfying \eqref{eq:f-linear}. Let $\Omega\subset\Rn$ be an open set and let $J_f$ be
the functional defined by \eqref{eq:Jf}. If $u\in BV(\Omega)$ and $\{u_j\}$ is any sequence in $BV(\Omega)$ such that
$u_j\to u$ in $L^1_\loc(\Omega)$, then
\[
 J_f(u;\Omega)\leq\liminf_{j\to+\infty} J_f(u_j;\Omega)\,.
\]
\end{theorem}

\begin{proof}[Proof of Theorem \ref{thm:ineq-bv}]
 We are going to prove a stronger inequality than \eqref{eq:ineq-bv}, i.e., 
 \begin{equation}\label{eq:ineq_bv-bis}
  J_f(\us;B\times\Ry)\leq J_f(u;B\times\Ry)\,,
 \end{equation}
for any Borel set $B\subset \proj(\Omega)$. As before we identify $u_0$ with $u$.

\step{1}  Let us first prove that $\us\in BV
(\omega\times\Ry)$ for every open set $\omega\Subset\proj(\Omega)$. Since $u\in BV_{0,y}(\Omega)$ then $u\in
BV(\omega\times\Ry)$. Hence, by approximation
%see e.g.\ \cite{EG}*{\S 5.2.2 Theorems 2 and 3} or \cite{AFP}*{\S 3.9 and 3.13}
we can find a sequence of non-negative functions $\{u_h\}\subset C^1(\omega\times\Ry)$ such that $u_h\to u$ in
$L^1(U\times\Ry)$ and 
\[\lim_{h\to\infty} \int_{\omega\times\Ry}\abs{\nabla u_h}\,dz=\abs{Du}(\Omega\times\Ry)\,.\]
By the continuity of the Steiner rearrangement --- equation \eqref{lemma:s_continuous} --- we get that
$(u_h)^\sigma\to\us$ in
$L^1(\omega\times\Ry)$; moreover by \eqref{eq:disug-diversa} we have that the sequence $\norm{\nabla
u_h^\sigma}_{L^1(\omega\times\Ry)}$ is bounded. Therefore (see, e.g., \cite{AFP}*{Theorem 3.9}) we conclude that $\us\in
BV(\omega\times\Ry)$.

\step{2} Let us assume, for the moment, that $u$ is compactly supported in $\Omega$. By Theorem \ref{thm:Su-}, $\Su$ is
a set of finite perimeter in
$\R^{n+1}$. On applying Proposition \ref{prop:J=F}, Theorem \ref{thm:F-ineq} and \eqref{eq:su_equiv_sus} we deduce that
for every Borel set $B\subset\proj(\Omega)$
\[\begin{split}
  J_f(\us;B\times\Ry)&=\int_{\brid \Sus \cap (B\times\Ry\times\Rt)} F_f(\nu^{\Sus})\,d\Hn\\
  &\leq \int_{\brid \Su \cap (B\times\Ry\times\Rt)} F_f(\nu^{\Su})\,d\Hn=J_f(u;B\times\Ry)\,,
 \end{split} 
\]
hence \eqref{eq:ineq_bv-bis} holds.

\step{3} Let us now drop the extra assumption. Fixed $\omega\Subset\proj(\Omega)$ we can find a smooth cutoff function
compactly supported in $\proj(\Omega)$ such that $\phi\equiv 1$ on $\omega$ and a smooth function $\eta$ compactly
supported in $\Rk$ with $\eta\equiv 1$ in $B(0,1)$. Let us define the functions
\[v(x,y)=u(x,y)\phi(x)\text{\ and\ }v_h(x,y)=v(x,y)\eta(\frac{y}{h})\,,\text{\ for\ }h\in\N\,.\]
Clearly, $v\in BV(\Rn)$ and $v_h\to v$ as $h\to +\infty$ in $L^1(\Rn)$. Hence, by Theorem \ref{thm:frelax} we deduce
that
\begin{equation}\label{eq:pr_BV_ineq-1}
 J_f(\us;\omega\times\Ry)=J_f(v^\sigma;\omega\times\Ry)\leq\liminf_{h\to+\infty}J_f(v_h^\sigma;\omega\times\Ry)\,.
\end{equation}
Moreover, since $\abs{D(v-v_h)}(\Rn)\to 0$ as $h\to +\infty$, we get
\begin{equation}\label{eq:pr_BV_ineq-2}
 \liminf_{h\to+\infty} J_f(v_h;\omega\times\Ry)=J_f(v;\omega\times\Ry)=J_f(u;\omega\times\Ry)\,.
\end{equation}
Now, for $B=\omega$ inequality \eqref{eq:ineq_bv-bis} follows from \eqref{eq:pr_BV_ineq-1}, \eqref{eq:pr_BV_ineq-2} and
the second step applied to $v_h$. Then, the general case where $B$ is any Borel set, is derived by approximation.
\end{proof}

\begin{proof}[Proof of Theorem \ref{thm:equal_bv}]
The proof is very similar to the one of Theorem \ref{thm:equality}. Thanks to \eqref{eq:su_equiv_sus}, it is sufficient
to show that $(\Su)^\sigma$ is equivalent to $\Su$.

\step{1} We claim that for $\Ha^{n-k+1}\aev (x,t)\in\esspj(\Su)$ there exists $R(x,t)>0$ such that the set
\[
 \{y:u(x,y)>t\}\text{\ is equivalent to }\{\abs{y}<R(x,t)\}\,.
\]
From \eqref{eq:equal_bv} and \eqref{eq:ineq_bv-bis} we see that equality holds in \eqref{eq:ineq_bv-bis} for any Borel
set $B\subset \proj(\Omega)$. Given any open set $\omega\Subset \proj(\Omega)$ let $\phi$ be a smooth cutoff function
with compact support in $\proj(\Omega)$ such that $\phi\equiv 1$ on $\omega$. Identifying $u$ with its extension $u_0$,
define $v:=u\phi$. Then, we have the following equality: 
\[
 J_f(v^\sigma;\omega\times\Ry)=J_f(v;\omega\times\Ry)\,.
\]
Hence, on using Proposition \ref{prop:J=F} we get
\[
 \int_{\brid\mathcal{S}_{v^\sigma}\cap(\omega\times\Ry\times\Rt)} F(\nu^{\mathcal{S}_{v^\sigma}})\,d\Hn
 =\int_{\brid\mathcal{S}_{v}\cap(\omega\times\Ry\times\Rt)} F(\nu^{\mathcal{S}_{v}})\,d\Hn\,.
\]
Since $v$ belongs to $BV(\Rn)$ and it is non-negative, from \eqref{eq:condizioni_om1} we deduce that $v$ has compact
support and therefore $\mathcal{S}_v$ has finite perimeter in $\R^{n+1}$. By the last equality and Lemma
\ref{lemma:f_F} below, the claim is proven from Proposition \ref{thm:equal_F} and from the arbitrariness of $\omega$.

\step{2} We have just proved that for $\Ha^{n-k+1}\aev (x,t)\in\esspjt(\Su)$ the $(x,t)$ section of $\Su$ is equivalent
to a ball in $\Rk$ with radius $R(x,t)$. Define $b:\Rnk\times\Rt\to\Rn$ to be the center of this ball. On applying Step
1 to the function $\us$ we see that for  $\Ha^{n-k+1}\aev (x,t)\in\esspjt(\Sus)$ every section
$(\Su)^\sigma_{x,t}$ is equivalent to a ball of the same radius $R(x,t)$ with center $\tilde{b}(x,t)$. From the
definition of
the
Steiner rearrangement we get $\tilde{b}(x,t)\equiv (x,0,t)$. Now the Theorem follows once we prove that
$b-\tilde{b}\equiv(0,c,0)$ for some $c\in\Rk$.

The case $k=1$ is \cite{CF}*{Theorem 2.5}. Let $k>1$ and denote by $S_i$ the Steiner symmetrization with respect to
$y_i$ for $i=1,\dotsc,k$. Since $\Omegas=(\Omegas)^{S_i}=(\Omega^{S_i})^\sigma$, from \eqref{eq:ineq-bv} we have the
following inequalities
\begin{equation}\label{eq:pr_bv_ug1}
 J_f(\us;\Omegas)\leq J_f(u^{S_i};\Omega^{S_i})\leq J_f(u;\Omega)\,.
\end{equation}
From the assumption \eqref{eq:equal_bv} we get equalities in \eqref{eq:pr_bv_ug1}.
Since almost every section $(\Su)_{x,t}$ is a ball, arguing as in Step 1 of the proof of Proposition
\ref{prop:nu_uguale} we get  
\[
\Ha^n\bigl(\{z\in\Omega:\partial_{y_i}u(z)=0\} \cap
\{z\in\Omega:\text{ either }M(z')=0 \text{ or } u(z)<M(z')\}\bigr)=0\,,
\]
where $z':=(x,y_1,\dotsc,y_{i-1},y_{i+1},\dotsc y_k)$. Similarly we also have
\[
 \Ha^{n-1}\bigl(\{z\in\brid\Omega : \nu^\Omega_{y_i}=0\}\cap\{\pi_{n-1}(\Omega)\times\R_{y_i}\}\bigr)=0\,,
\]
where $\pi_{n-1}$ is the projection on $z'$. Since $\Omega^\sigma=(\Omegas)^{S_1}$, by the $k=1$ case, we have that
$(b(x,t))_{y_1}\equiv c_1$ for some $c_1\in\R$.
Now iterate the procedure and obtain $(b(x,t))_y\equiv(c_1,\dotsc,c_k)$ and so $b-\tilde{b}\equiv(0,c,0)$ with
$c=(c_1,\dotsc,c_k)$.
\end{proof}

The following lemma shows how properties of the function $f$ are inherited by $F_f$.
\begin{lemma}[(\cite{CF}*{Lemma 6.1}]\label{lemma:f_F}
 Let $f:\Rn\to[0,+\infty)$ be a convex function vanishing at $0$. Then, the functions $F_f$ defined by \eqref{eq:F_f}
is a convex function satisfying \eqref{eq:F-hom}. Moreover, if in addition $f$ is as in Theorem \ref{thm:equal_bv}, then
$F_f$ satisfies \eqref{eq:F-radial}, \eqref{eq:Fincr} and \eqref{eq:Kconvex} with $K=\Rnk\times(\Rt^-\cup\{0\})$.
\end{lemma}

\begin{remark}\label{rk:fsobo}
 Here we want to observe that if $f$ is a non-negative function as in Theorem \ref{thm:equality}, then the function
$F_f(\xi_1,\dotsc,\xi_{n+1})$, possibly attaining infinite value if $\xi_{n+1}\geq 0$, defined as in \eqref{eq:F_f}
satisfies the assumptions of Proposition \ref{thm:equal_F}. However, if $u\in\sobo(\Omega)$ then \eqref{eq:J=F} still
holds and thus Lemma \ref{lemma:sezioni_palle} follows arguing as in Step 1 of the proof of Theorem \ref{thm:equal_bv}.
\end{remark}

\section*{Acknowledgements}
This research was funded by the
2008 ERC Advanced Grant no.\ 226234 \emph{Analytic Techniques for
Geometric and Functional Inequalities}.

The author would like to thank M.\ Barchiesi and L.\ Brasco for some helpful discussions.

\bibliography{steiner}
\end{document}